\documentclass[nolinenbrs]{dmgt}

\usepackage{amsmath}
\usepackage{amssymb}
\usepackage{amsthm}
\usepackage{tikz}
\usepackage{caption}
\usepackage{subcaption}
\usepackage{cleveref}

\tikzstyle{vertex}=[circle,fill=black!100,minimum size=6pt,inner sep=0pt]
\tikzstyle{edge} = [draw,thick,-]
\tikzstyle{dashed-edge} = [draw,thick,-, dashed]
\tikzstyle{very-thick-edge} = [draw,line width=1mm,-]
\tikzstyle{bold-edge} = [draw,line width=0.6mm,-]
\tikzstyle{dotted-edge} = [draw,thick,-, dotted]
\tikzstyle{bold-dotted-edge} = [draw,line width=0.6mm,-, dotted]
\tikzstyle{possible-path} = [edge, loosely dotted]
\tikzstyle{very-thick-possible-path} = [very-thick-edge, dotted]
\tikzstyle{color-label}=[font=\small,inner sep=0pt]
\tikzstyle{graph-label}=[font=\small,inner sep=0pt]

\def\ftnkt{\footnote{Krzysztof Turowski's research was funded in whole by Polish National Science Center $2020$/$39$/D/ST$6$/$00419$ grant. For the purpose of Open Access, the author has applied a CC-BY public copyright license to any Author Accepted Manuscript (AAM) version arising from this submission.}}
\def\ftnbw{\footnote{Corresponding author}}

\newauthor{Robert Janczewski}{R. Janczewski}{
Department of Algorithms and System Modeling\\
Faculty of Electronics, Telecommunication and Informatics\\
Gdańsk University of Technology\\
ul.\ Narutowicza $11/12$, Gdańsk, Poland
}[skalar@eti.pg.gda.pl]

\newauthor{Krzysztof Turowski\ftnkt}{K. Turowski}{
Theoretical Computer Science Department\\
Faculty of Mathematics and Computer Science\\
Jagiellonian University\\
ul.\ Łojasiewicza $6$, Kraków, Poland
}[krzysztof.szymon.turowski@gmail.com]

\newauthor{Bartłomiej Wróblewski\ftnbw}{B. Wróblewski}{
Department of Algorithms and System Modeling\\
Faculty of Electronics, Telecommunication and Informatics\\
Gdańsk University of Technology\\
ul.\ Narutowicza $11/12$, Gdańsk, Poland
}[bart.wroblew@gmail.com]

\title{Edge coloring of products of signed graphs}

\keywords{signed graphs, edge coloring of signed graphs, graph products}

\classnbr{05C15}

\begin{document}
\begin{abstract}
    In 2020, Behr defined the problem of edge coloring of signed graphs and showed that every signed graph $(G, \sigma)$ can be colored using exactly $\Delta(G)$ or $\Delta(G) + 1$ colors, where $\Delta(G)$ is the maximum degree in graph $G$.

    In this paper, we focus on products of signed graphs. We recall the definitions of the Cartesian, tensor, strong, and corona products of signed graphs and prove results for them. In particular, we show that $(1)$ the Cartesian product of $\Delta$-edge-colorable signed graphs is $\Delta$-edge-colorable, $(2)$ the tensor product of a $\Delta$-edge-colorable signed graph and a signed tree requires only $\Delta$ colors and $(3)$ the corona product of almost any two signed graphs is $\Delta$-edge-colorable. We also prove some results related to the coloring of products of signed paths and cycles.
\end{abstract}

\section{Introduction}

    This article focuses on simple, finite, and undirected graphs. A graph $G$ is defined by its vertices $V(G)$ and edges $E(G)$. Their respective quantities are denoted by $n(G)$ and $m(G)$. The degree of a vertex $v$ in $G$ is given by $\deg_G(v)$ (or sometimes just $\deg(v)$, when $G$ is unambiguous from the context), and $\Delta(G)$ indicates the maximum degree among all vertices in $G$.
    
    An \emph{incidence} consists of a vertex $v$ paired with an edge $e$, with $v$ acting as an endpoint of $e$. This pairing is denoted $v\text{:}e$, and the set of all such incidences for a graph $G$ is denoted by $I(G)$. Any other terms and definitions are consistent with those put forth by Diestel \cite{diestel}.
    
    A concept of signed graphs emerged in the $1950$s, introduced by Harary \cite{harary_signed}. They represent an extended form of simple graphs, devised to better model social relations. A \emph{signed graph} is a pair $(G, \sigma)$, where $G$ is a graph and $\sigma\colon E(G)\to\{\pm 1\}$ is a function. Here, $G$ is an \emph{underlying graph}, and $\sigma$ denotes its \emph{signature}. An underlying graph of a signed graph $S$ is sometimes also denoted by $|S|$. An edge $e\in E(G)$ is called \emph{positive} (or \emph{negative}) if $\sigma(e)=1$ (or $\sigma(e)=-1$, respectively). Cycles in $(G, \sigma)$ are called \emph{positive} or \emph{negative} based on the product of their edge signs. A signed graph having only positive cycles is \emph{balanced}. In all other cases, it is \emph{unbalanced}.
    
    In the context of a signed graph $(G, \sigma)$, switching refers to an operation on a subset $V'\subseteq V(G)$, yielding a new signed graph $(G, \sigma')$. For any given edge $uv$ in $E(G)$, the following is the result of switching:
    \begin{displaymath}
        \sigma'(uv) = 
        \begin{cases} 
            -\sigma(uv)\text{,} & \text{if } V' \text{ includes exactly one of } u\text{, }v\text{,} \\
            \sigma(uv)\text{,} & \text{otherwise.} 
        \end{cases}
    \end{displaymath}
    If a graph $S'$ can be obtained from graph $S$ by switching some vertices of $S$, we call them \emph{switching equivalent} signed graphs.

    Additionally, by \emph{decomposition} of a signed graph $S = (G, \sigma)$ we denote a sequence of signed graphs $S_1 = (G_1, \sigma_1)$, \ldots, $S_k = (G_k, \sigma_k)$ such that $E(G) = \bigcup_{i = 1}^k E(G_i)$, $E(G_i) \cap E(G_j) = \emptyset$ for any $1 \le i < j \le k$, and $\sigma_i(e) = \sigma(e)$ for all $e \in E(G_i)$ and for all $i \in \{1, \ldots, k\}$. It turns out that it is often useful to decompose signed graphs into (properly chosen) sequences of signed cycles and/or matchings. 

    There were several coloring problems on graphs considered in the literature, starting with works of Harary \cite{harary1968coloring} and Zaslavsky \cite{zaslavsky1982signed,zaslavsky1984colorful,zaslavsky1987balanced}. In particular, there is a generalized definition of the chromatic number, for which on the one hand the equivalent of Vizing's theorem was obtained by Mácajová et al. \cite{mavcajova2016chromatic}, but also it was proved that the Four Color Theorem for signed planar graphs does not hold \cite{kardovs20214}.
    In addition, there were also studied other problems on signed graphs such as circular coloring \cite{kang2018circular}, choosability \cite{jin2016choosability,schweser2017degree}, or chromatic spectrum \cite{kang2016chromatic,kang2018hajos}.
     For an overview of these research directions in the recent years see \cite{steffen2021concepts}.

    Recently, Behr \cite{behr} defined the problem of edge coloring of signed graphs in such a way that it generalizes the well-known edge coloring problem. Let $k$ be a positive integer and
    \begin{displaymath}
        M_k=
        \begin{cases}
            \{0, \pm 1, \ldots, \pm l\}, & \text{if $k=2l+1$,}\\
            \{\pm 1, \ldots, \pm l\}, & \text{if $k=2l$.}
        \end{cases}
    \end{displaymath}
    A \emph{$k$-edge-coloring} of a signed graph $S = (G, \sigma)$ is a function $f\colon I(G)\to M_k$ such that $f(u\text{:}uv)=-\sigma(uv)f(v\text{:}uv)$ for each edge $uv\in E(G)$ and $f(u\text{:}uv_1)\neq f(u\text{:}uv_2)$ for edges $uv_1\neq uv_2$. \emph{The chromatic index} of a signed graph $S$ is denoted by $\chi'(S)$ and defined as the smallest $k$ for which graph $S$ has a $k$-edge-coloring. It is a well-known fact that all the switching equivalent signed graphs have exactly the same chromatic index. 

    Behr \cite{behr} also proved the generalized Vizing's theorem, also called the Behr's theorem:
    \begin{theorem}[\cite{behr}]
        $\Delta(|S|)\leq\chi'(S)\leq\Delta(|S|)+1$ for every signed graphs $S$. \qed
    \end{theorem}

    Behr \cite{behr} introduced a \emph{class ratio $\mathcal{C}(G)$} of graph $G$ as the number of signatures $\sigma\colon E(G)\to\{\pm 1\}$ such that a signed graph $(G, \sigma)$ can be colored using $\Delta(G)$ colors, divided by the number of all possible signatures that can be defined on $E(G)$, that is, $2^{m(G)}$. The class ratio is a rational number satisfying $0\leq\mathcal{C}(G)\leq 1$.

    Additionally, Behr \cite{behr} proved some basic results concerning the problem of edge coloring of paths and cycles:
    \begin{theorem}[\cite{behr}]
        \label{thm:behr-path}
        Every signed path $S = (P_k, \sigma)$ can be colored using exactly $\Delta(|S|)$ colors. \qed
    \end{theorem}

    It is easy to observe that a signed graph of a maximum degree $1$ can be colored using exactly one color---$0$. 

    \begin{theorem}[\cite{behr}]
        \label{thm:behr-cycle}
        A signed cycle $S = (C_k, \sigma)$ can be colored using $\Delta(|S|)$ colors if and only if $S$ is balanced. Otherwise, it requires $\Delta(|S|) + 1$ colors. \qed
    \end{theorem}

    It is also known that:
    \begin{theorem}[\cite{classes_one_two}]
        Every signed cactus $S$ that is not an unbalanced cycle can be colored using exactly $\Delta(|S|)$ colors.  \qed
    \end{theorem}

    \begin{cor}
        Every signed tree $S$ can be colored using exactly $\Delta(|S|)$ colors. \qed
    \end{cor}

    If $\chi'(S) = \Delta(|S|)$, we say $S$ is $\Delta$-edge-colorable.
    
    It follows from \Cref{thm:behr-path} and \Cref{thm:behr-cycle} that a signed graph of maximum degree $2$ can be colored using $2$ colors if each of its connected components is either a path or a balanced cycle. In most of the proofs, we show that a given signed graph has a decomposition into subgraphs of maximum degree $2$ or graphs of maximum degree $2$ and a graph of maximum degree $1$. Then it is clear that the subgraphs can be properly colored using either $2$ or $1$ colors and the colors used to color specific incidences do not matter for the proofs. We always guarantee that the sets of colors used to color different subgraphs are disjoint.

    In this paper, we focus on the products of signed graphs and determine the value of the chromatic index of products of some signed graphs. In \Cref{section_cartesian} we define what the Cartesian product of signed graphs is and prove that a Cartesian product of two $\Delta$-edge-colorable signed graphs is also $\Delta$-edge-colorable. We also find the chromatic index of products of a signed path and a signed cycle, and products of two signed cycles. In \Cref{section_tensor} we define a tensor product of signed graphs and prove that a tensor product of a $\Delta$-edge-colorable signed graph and a signed tree is $\Delta$-edge-colorable. In \Cref{section_strong} we define a strong product of signed graphs and prove that a strong product of two signed paths can be colored using $\Delta$ colors. In \Cref{section_corona} we deal with the concept of corona products of signed graphs by first giving its definition and then proving that the chromatic index of a corona product of any two non-trivial signed graphs is exactly equal to $\Delta$.

    Throughout the paper for simplicity we will often switch between signed graphs and their underlying graphs e.g. refer to $V(S)$, $E(S)$, $n(S)$, $m(S)$, $\Delta(S)$ instead of $V(|S|)$ etc.
    In the same manner, whenever we will refer to a sequence $H_1$, \ldots, $H_k$ as a decomposition of $S$, we will be denoting both signed graphs or their underlying graphs, depending on the particular context.

\section{Cartesian products}\label{section_cartesian}

    Let $S_1 = (G_1, \sigma_1)$ and $S_2 = (G_2, \sigma_2)$ be signed graphs. The Cartesian product of graphs $S_1$, $S_2$ is a signed graph $S = (G, \sigma)$, such that $V(S) = V(S_1) \times V(S_2)$ and there is an edge $e = (u, u') (v, v')$ in $S$, where $u, v \in V(S_1)$, $u', v' \in V(S_2)$, if and only if one of the following conditions holds:
    \begin{enumerate}
        \item $u = v$, $u' v' \in E(S_2)$. In such case $\sigma(e) = \sigma_2(u' v')$;
        \item $u v \in E(S_1)$, $u' = v'$. In such case $\sigma(e) = \sigma_1(u v)$.
    \end{enumerate}
    
    We denote $S$ by $S_1 \square S_2$. Clearly, $|S| = G_1 \square G_2$.

    It is easy to observe that $S$ consists of $n(S_1)$ disjoint copies of graph $S_2$ and $n(S_2)$ disjoint copies of graph $S_1$. It is also clear that $\deg_S((u, u')) = \deg_{S_1}(u) + \deg_{S_2}(u')$. It follows that $\Delta(S) = \Delta(S_1) + \Delta(S_2)$.
    
    \begin{theorem}\label{cartesian_delta_delta}
        Let $S_1 = (G_1, \sigma_1)$, $S_2 = (G_2, \sigma_2)$ be edge-disjoint signed graphs and $S = (G, \sigma) = S_1 \square S_2$. If $\chi'(S_1) = \Delta(S_1)$ and $\chi'(S_2) = \Delta(S_2)$, then $\chi'(S) = \Delta(S)$.
    \end{theorem}
    
    \begin{proof}[Proof of theorem \ref{cartesian_delta_delta}]
        Let $V(S_i) = \{v_1^i, \ldots, v_{n_i}^i\}$, $E(S_i) = \{e_1^i, \ldots, e_{m_i}^i\}$ for $i \in \{1, 2\}$.
        
        We define functions $\eta_i^V \colon V(S) \to V(S_i)$ for $i \in \{1, 2\}$, such that $\eta_1^V(v_i^1, v_j^2) = v_i^1$ and $\eta_2^V(v_i^1, v_j^2) = v_j^2$. Analogously, we define functions $\eta_i^E \colon E(S) \to E(S_i)$ for $i \in \{1, 2\}$, such that $\eta_1^E((v_i^1, v_j^2),\allowbreak (v_k^1, v_j^2)) = v_i^1 v_k^1$, $\eta_2^E((v_i^1, v_j^2), (v_i^1, v_k^2)) = v_j^2 v_k^2$ and functions $\eta_i^I \colon I(S) \to I(S_i)$ for $i \in \{1, 2\}$, such that $\eta_i^I(v\text{:}e)) = \eta_i^V(v)\text{:}\eta_i^E(e)$.
        
        Let $c_i$ be an arbitrary optimal edge coloring of graph $S_i$, $i \in \{1, 2\}$.
        
        We consider two cases:
        \begin{enumerate}
            \item At least one of $\Delta(S_1)$, $\Delta(S_2)$ is even. Without loss of generality, we assume that $\Delta(S_2)$ is even. Let $c$ be a function such that:

            \begin{displaymath}
                c(v \text{:} e) =
                \begin{cases}
                    c_1(\eta_1^I(v \text{:} e)), & \text{if } e = (v_i^1, v_j^2), (v_k^1, v_j^2)\text{;}
                    \\
                    c_2(\eta_2^I(v \text{:} e)) + \lfloor \frac{\Delta(S_1)}{2} \rfloor, & \text{if } e = (v_i^1, v_j^2), (v_i^1, v_k^2) \text{ and } c_2(\eta_2^I(v \text{:} e)) > 0\text{;}
                    \\
                    c_2(\eta_2^I(v \text{:} e)) - \lfloor \frac{\Delta(S_1)}{2} \rfloor, & \text{if } e = (v_i^1, v_j^2), (v_i^1, v_k^2) \text{ and } c_2(\eta_2^I(v \text{:} e)) < 0\text{.}
                \end{cases}
            \end{displaymath}
        
            The equation covers all possible incidences because function $c_2$ is never equal to $0$ since $\Delta(S_2)$ is even and $c_2$ is a $\Delta(S_2)$-edge coloring of graph $S_2$.
            
            We observe that all the incidences from the copies of graph $S_1$ get the same colors in $c$ as the respective incidences in $c_1$. The copies are disjoint so the incidences are colored correctly. We denote those colors by $C_1$. $C_1 = \{\pm 1, \ldots, \pm\frac{\Delta(S_1)} {2}\}$ if $\Delta(S_1)$ is even or $C_1 = \{0, \pm1, \ldots, \pm\frac{\Delta(S_1)}{2}\}$ if $\Delta(S_1)$ is odd.
            
            We observe that incidences from the copies of graph $S_2$ get colors from the set $C_2 = \{1 + \lfloor \frac{\Delta(S_1)} {2}\rfloor, \ldots, \frac{\Delta(S_2)} {2} + \lfloor \frac{\Delta(S_1)} {2}\rfloor\} \cup \{-1 - \lfloor \frac{\Delta(S_1)} {2}\rfloor, \ldots, -\frac{\Delta(S_2)} {2} - \lfloor \frac{\Delta(S_1)} {2}\rfloor\}$ in $c$. The copies are disjoint and the colors used to color their incidences are equal to colors from $c_2$ shifted by either $\lfloor \frac{\Delta(S_1)}{2} \rfloor$ or $-\lfloor \frac{\Delta(S_1)}{2} \rfloor$, so the incidences are colored correctly.
            
            It is clear that $C_1 \cap C_2 = \emptyset$ since $\max\{|x|\colon x \in C_1\} < \min\{|x|\colon x \in C_2\}$. It follows that $c$ is an edge coloring of graph $S$. We observe that $c$ uses exactly $2 (\frac{\Delta(S_2)} {2} + \lfloor \frac{\Delta(S_1)} {2}\rfloor) = \Delta(S_1) + \Delta(S_2) = \Delta(S)$ colors when $\Delta(S_1)$ is even and exactly $1 + 2 (\frac{\Delta(S_2)} {2} + \lfloor \frac{\Delta(S_1)} {2}\rfloor) = 1 + 2 (\frac{\Delta(S_2)} {2} + \frac{\Delta(S_1) - 1} {2}) = \Delta(S_1) + \Delta(S_2) = \Delta(S)$ colors when $\Delta(S_1)$ is odd.
            
            \item Both $\Delta(S_1)$ and $\Delta(S_2)$ are odd. We observe that if we followed the previous case, both $c_1$ and $c_2$ would use color $0$ so the previous reasoning does not directly apply in this case since there would be adjacent incidences both colored with color $0$. We also observe that $\Delta(S)$ is even, so if there exists a $\Delta(S)$-coloring of graph $S$, it does not use color $0$. We show how to modify the reasoning so it also applies here.
            
            Let $c$ be a function analogous to the function $c$ from the previous case. We only extend it such that $c(v \text{:} e) = 0$ if $e = (v_i^1, v_j^2) (v_i^1, v_k^2)$ and $c_2(\eta_2^I(v\text{:}e)) = 0$. That way, it has a value for all the possible incidences of graph $S$. We observe that $c$ may not be an edge coloring of $S$, because there might be adjacent incidences sharing the same vertex with color $0$ assigned to both of them.
            
            Let $H$ be a subgraph of $S$ such that $uv \in E(H)$ if and only if $c(u \text{:} uv) = 0$ and $c(v \text{:} uv) = 0$, that is, it contains exactly the edges of $S$ such that their incidences get color $0$ in $c$. Note that from the definition of function $c$, it follows that if for some edge $uv$ it is true that $c(u \text{:} uv) = 0$, it is also true that $c(v\text{:}uv) = 0$. It is clear that $H$ is not an empty graph, because there exists at least one incidence in $S_i$ such that it gets color $0$ in coloring $c_i$, for $i \in \{1, 2\}$.
            
            Let us consider an arbitrary connected component $H'$ of graph $H$. Without loss of generality, we assume that $(v_i^1, v_j^2)$ $(v_k^1, v_j^2) \in E(H')$. We denote that edge by $e_1$. It is clear that $H'$ does not contain edge $(v_i^1, v_j^2)$ $(v_p^1, v_j^2)$ for $p \neq k$ and edge $(v_t^1, v_j^2)$ $(v_k^1, v_j^2)$ for $t \neq i$ for otherwise, there would be adjacent edges in $S_1$ both with color $0$ assigned in $c_1$. We define set $E'$ consisting of two edges: $e_2 = (v_i^1, v_j^2)$ $(v_i^1, v_p^2)$ and $e_3 = (v_k^1, v_p^2)$ $(v_k^1, v_j^2)$. There are two cases to consider:
            \begin{enumerate}
                \item $E' \cap E(H') = \emptyset$. Then $m(H') = 1$ so $H'$ is a matching.
                \item $E' \cap E(H') \neq \emptyset$. We observe that $E' \subset E(H')$ because $\eta_2^E(e_2) = \eta_2^E(e_3)$. Let $e_4 = (v_i^1, v_p^2), (v_k^1, v_p^2)$. Since $\eta_1^E(e_1) = \eta_1^E(e_4)$ and $e_4$ is adjacent to $e_2$ and $e_3, e_4 \in E(H')$.
                
                It is easy to observe that $H'$ contains only edges $e_1$, \ldots, $e_4$. Otherwise, edges colored by either $c_1$ or $c_2$ with color $0$ would not be matchings, so either $c_1$ or $c_2$ would not be a proper coloring. We observe that edges of graph $H'$ span a cycle $C_4$ and since $\sigma(e_1) = \sigma(e_4)$, $\sigma(e_2) = \sigma(e_3)$, the cycle is balanced.
            \end{enumerate}
        
            It follows that all the connected components of graph $H$ are either matchings or balanced cycles $C_4$, so $H$ can be colored using $2$ colors. We change the values of function $c$ such that incidences previously colored by color $0$ are now colored using colors $\pm(\lfloor \frac{\Delta(S_1)}{2} \rfloor + \lfloor \frac{\Delta(S_2)}{2} \rfloor + 1)$. After this change, $c$ is an edge coloring of graph $S$. Moreover, $c$ uses exactly $2 (\lfloor \frac{\Delta(S_1)}{2} \rfloor + \lfloor \frac{\Delta(S_2)}{2} \rfloor + 1) = 2 ( \frac{\Delta(S_1) - 1}{2} + \frac{\Delta(S_2) - 1}{2} + 1) = \Delta(S_1) + \Delta(S_2) = \Delta(S)$, so $c$ is a proper $\Delta(S)$-edge coloring of graph $S$.
        \end{enumerate}
    \end{proof}
    
    \noindent Since every signed path is $\Delta$-edge-colorable, it follows that:
    
    \begin{cor}\label{cartesian_path_path}
        Let $S_1 = (P_r, \sigma_1)$, $S_2 = (P_s, \sigma_2)$ and $S = S_1
        \square S_2$. Then it follows that $\chi'(S) = \Delta(S)$. \qed
    \end{cor}
    
    \begin{figure}
        \centering
        \begin{tikzpicture}[thick, scale=0.7]
            \node[graph-label] (label) at (2,-0.75) {a)};
            \node[vertex] (1) at (0, 0) {};
            \node[vertex] (2) at (2, 0) {};
            \node[vertex] (3) at (4, 0) {};
            
            \node[vertex] (5) at (0, 2) {};
            \node[vertex] (6) at (2, 2) {};
            \node[vertex] (7) at (4, 2) {};
            
            \node[vertex] (9) at (0, 4) {};
            \node[vertex] (10) at (2, 4) {};
            \node[vertex] (11) at (4, 4) {};
            
            \node[vertex] (13) at (0, 6) {};
            \node[vertex] (14) at (2, 6) {};
            \node[vertex] (15) at (4, 6) {};

            \foreach \source / \dest in {1/2,2/3,5/6,6/7,9/10,10/11,13/14,14/15}
                    \path [dotted-edge] (\source) -- (\dest);
                    
            \foreach \source / \dest in {1/5,9/13,2/6,10/14,3/7,11/15}
                    \path [bold-edge] (\source) -- (\dest);

            \foreach \source / \dest in {1/2,13/14}
                    \path [bold-edge] (\source) -- (\dest);
                    
            \foreach \source / \dest in {5/9,6/10,7/11}
                    \path [bold-edge] (\source) -- (\dest);

            \foreach \source / \dest in {1/13,2/14,3/15}
                    \path [dotted-edge] (\source) edge [bend left=30] (\dest);

        \end{tikzpicture}
        \hspace{1cm}%
        \begin{tikzpicture}[thick, scale=0.7]
            \node[graph-label] (label) at (3,-0.75) {b)};
            \node[vertex] (1) at (0, 0) {};
            \node[vertex] (2) at (2, 0) {};
            \node[vertex] (3) at (4, 0) {};
            \node[vertex] (4) at (6, 0) {};
            
            \node[vertex] (5) at (0, 2) {};
            \node[vertex] (6) at (2, 2) {};
            \node[vertex] (7) at (4, 2) {};
            \node[vertex] (8) at (6, 2) {};
            
            \node[vertex] (9) at (0, 4) {};
            \node[vertex] (10) at (2, 4) {};
            \node[vertex] (11) at (4, 4) {};
            \node[vertex] (12) at (6, 4) {};
            
            \node[vertex] (13) at (0, 6) {};
            \node[vertex] (14) at (2, 6) {};
            \node[vertex] (15) at (4, 6) {};
            \node[vertex] (16) at (6, 6) {};

            \foreach \source / \dest in {1/2,2/3,3/4,5/6,6/7,7/8,9/10,10/11,11/12,13/14,14/15,15/16}
                    \path [dotted-edge] (\source) -- (\dest);
                    
            \foreach \source / \dest in {1/5,9/13,2/6,10/14,3/7,11/15,4/8,12/16}
                    \path [bold-edge] (\source) -- (\dest);

            \foreach \source / \dest in {1/2,3/4,13/14,15/16}
                    \path [bold-edge] (\source) -- (\dest);
                    
            \foreach \source / \dest in {5/9,6/10,7/11,8/12}
                    \path [bold-edge] (\source) -- (\dest);

            \foreach \source / \dest in {1/13,2/14,3/15,4/16}
                    \path [dotted-edge] (\source) edge [bend left=30] (\dest);

        \end{tikzpicture}
        \caption{Example signed graphs considered in \Cref{cartesian_path_cycle} -- $(a)$ $(P_3 \square C_4, \sigma)$, $(b)$ $(P_4 \square C_4, \sigma)$. Edges belonging to graph $H_1$ are marked with bold lines in both cases.}
        \label{figure_cartesian_path_cycle}
    \end{figure}
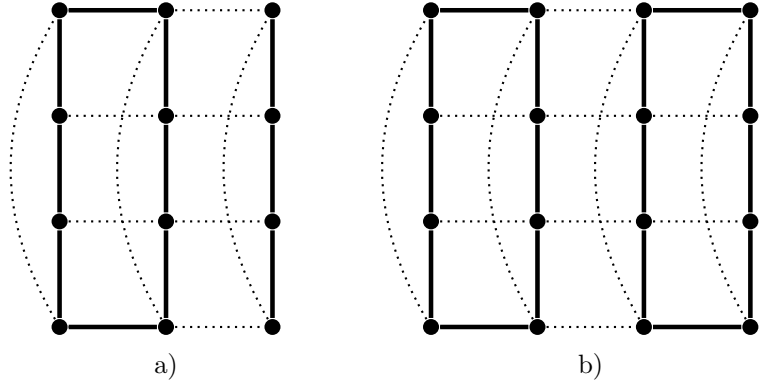

    Next, we deal with Cartesian products of paths and cycles. Obviously, since $P_1 \square C_s = C_s$ and the chromatic index of signed cycles is already known (see \Cref{thm:behr-cycle}), we can restrict our attention to longer paths:
    \begin{theorem}\label{cartesian_path_cycle}
        Let $S_1 = (P_r, \sigma_1)$, $S_2 = (C_s, \sigma_2)$ and $S = (P_r \square C_s, \sigma) = S_1 \square S_2$. If $r > 1$, then $\chi'(S) = \Delta(S)$.
    \end{theorem}
    
    \begin{proof}[Proof of theorem \ref{cartesian_path_cycle}]
        We first observe that $\Delta(S) = 3$ if $r = 2$, otherwise $\Delta(S) = 4$. 
        
        Let $V(S_1) = \{v_1^1, \ldots, v_{r}^1\}$ and $V(S_2) = \{v_1^2, \ldots, v_{s}^2\}$.
        
        We first observe that $\sigma((v_i^1, v_j^2) (v_{i+1}^1, v_j^2)) = \sigma((v_i^1, v_k^2)(v_{i+1}^1, v_k^2))$ for $1 \leq i < r$, $1 \leq j, k \leq s$. Moreover, $\sigma((v_i^1, v_k^2)(v_i^1, v_{k + 1}^2)) = \sigma((v_j^1, v_k^2)(v_j^1, v_{k + 1}^2))$ for $1 \leq i, j \leq r$, $1 \leq k < s$ and $\sigma((v_i^1, v_1^2)(v_i^1, v_{s}^2)) = \sigma((v_j^1, v_1^2)(v_j^1, v_{s}^2))$ for $1 \leq i, j \leq r$.
        
        We will show the decomposition of $S$ into graphs $H_1$, $H_2$ such that both of them can be colored using $2$ colors. Let us first describe the structure of $H_1$: $V(H_1) = V(S)$, $E(H_1) = \{(v_i^1, v_j^2)(v_i^1, v_{j+1}^2) \colon 1 \leq i \leq r$, $1 \leq j < s\} \cup \{(v_{2i - 1}^1, v_j^2)(v_{2i}^1, v_j^2) \colon 1 \leq i \leq \lfloor \frac{r}{2} \rfloor$, $j \in
        \{1$, $s\} \}$ (see also examples of graph $H_1$ in \Cref{figure_cartesian_path_cycle}). It follows that $H_1$ contains $\lfloor \frac{r}{2}\rfloor$ cycles and possibly a single path when $r$ is odd. It is clear that all the cycles are balanced, as they contain an even number of pairs of edges with the same signs (it follows from the previous observation). Then it becomes clear that $H_1$ contains balanced cycles and possibly a path, so $H_1$ can be colored using $2$ colors.
        
        It is easy to verify that $H_2$ is then defined as $V(H_2) = V(S)$ with $E(H_2) = E_1 \cup E_2 \cup E_3$, where $E_1 = \{(v_i^1, v_1^2)(v_i^1, v_{s}^2) \colon 1 \leq j \leq r\}$, $E_2 = \{(v_i^1, v_j^2)(v_{i+1}^1, v_j^2) \colon \allowbreak 1 \leq i < r,\allowbreak 1 < j < s\}$ and $E_3 = \{(v_{2i}^1, v_j^2)(v_{2i + 1}^1, v_j^2) \colon 1 \leq i < \lfloor \frac{r}{2} \rfloor$, $j \in \{1, s\}\}$. It is clear that sets $E_1$, $E_2$, $E_3$ are disjoint. We observe that $E_2$ can be partitioned into $s - 2$ paths of length $r - 1$. $E_1 \cup E_3$ can be partitioned into 1 or 2 edges (when $r$ is odd or even, respectively) and $\lfloor \frac{r - 1}{2} \rfloor$ cycles. The cycles are balanced since each of them consists of two pairs of edges with the same signs. We consider two cases:
        \begin{enumerate}
            \item $r = 2$. $H_2$ has a decomposition into the following components: two disjoint edges and $s - 2$ disjoint paths of length $1$, which are just disjoint edges. $H_2$ has a decomposition into disjoint edges, so it is a matching and can be colored using color $0$. Since $H_1$ is colored using $2$ colors, $0$ is not used for $H_1$ and can be used to color $H_2$. It follows that $S$ can be colored using $3$ colors, and $\Delta(S) = 3$.
            \item $r > 2$. $H_2$ has a decomposition into the following components: one or two disjoint edges (when $r$ is odd or even, respectively), and $s - 2$ disjoint paths and balanced cycles. It follows that $H_2$ can be colored using $2$ colors, so $S$ can be colored using $4$ colors in total and $\Delta(S) = 4$.
        \end{enumerate}

        It follows that in all cases $S$ can be colored using $\Delta(S)$ colors in total.
    \end{proof}
    
    \begin{lemma}\label{cartesian_c2r_c2s}
        Let $S_1 = (C_{2r}, \sigma_1)$, $S_2 = (C_{2s}, \sigma_2)$ and $S = (G, \sigma) = S_1 \square S_2$. Then it holds that $\chi'(S) = \Delta(S)$.
    \end{lemma}
    
    \begin{proof}
        First, let us note that $\Delta(S) = 4$. Let $V(S_1) = \{v_1^1, \ldots, v_{2r}^1\}$ and $V(S_2) = \{v_1^2, \ldots, v_{2s}^2\}$.
        
        We will show the decomposition of $S$ into graphs $H_1$, $H_2$ such that both of them can be colored using $2$ colors regardless of the edge signs. We first describe the structure of $H_1$. Obviously, $V(H_1) = V(S)$. Graph $H_1$ contains:
        \begin{enumerate}
            \item\label{case_1_cartesian_c2r_c2s} Edges $e_1 = (v_1^1, v_1^2) (v_1^1, v_{2s}^2)$, $e_2 = (v_1^1, v_{2s}^2) (v_{2r}^1, v_{2s}^2)$, $e_3 = (v_{2r}^1, v_{2s}^2) (v_{2r}^1,\allowbreak v_1^2)$, $e_4 = (v_{2r}^1, v_1^2) (v_1^1, v_1^2)$. It is easy to observe that these edges span a cycle (see \Cref{fig:thm9-1}) and it is balanced regardless of the signs in $S_1$ and $S_2$ since $\sigma(e_1) = \sigma(e_3)$ and $\sigma(e_2) = \sigma(e_4)$.

            \item\label{case_2_cartesian_c2r_c2s} Edges $e_1^i = (v_{2i}^1, v_1^2) (v_{2i}^1, v_{2s}^2)$, $e_2^i = (v_{2i}^1, v_{2s}^2) (v_{2i + 1}^1, v_{2s}^2)$, $e_3^i = (v_{2i + 1}^1, \allowbreak v_{2s}^2) (v_{2i + 1}^1, v_1^2)$, $e_4^i = (v_{2i + 1}^1, v_1^2) (v_{2i}^1, v_1^2)$ for $1 \leq i < r$. Again, it is easy to observe that $e_1^i$, $e_2^i$, $e_3^i$, $e_4^i$ span a cycle (see \Cref{fig:thm9-2}) and it is balanced since $\sigma(e_1^i) = \sigma(e_3^i)$ and $\sigma(e_2^i) = \sigma(e_4^i)$ for $1 \leq i < r$. All the cycles are vertex-disjoint.

            \item\label{case_3_cartesian_c2r_c2s} Edges $e_1^i = (v_{1}^1, v_{2i}^2) (v_{1}^1, v_{2i + 1}^2)$, $e_2^i = (v_{1}^1, v_{2i + 1}^2) (v_{2r}^1, v_{2i + 1}^2)$, $e_3^i = (v_{2r}^1,$ $v_{2i + 1}^2) (v_{2r}^1, v_{2i}^2)$, $e_4^i = (v_{2r}^1, v_{2i}^2) (v_{1}^1, v_{2i}^2)$ for $1 \leq i < s$. We note that $e_1^i$, $e_2^i$, $e_3^i$, $e_4^i$ span a cycle (see \Cref{fig:thm9-3}) and it is balanced since $\sigma(e_1^i) = \sigma(e_3^i)$ and $\sigma(e_2^i) = \sigma(e_4^i)$ for $1 \leq i < s$. All the cycles are vertex-disjoint.
            
            \item\label{case_4_cartesian_c2r_c2s} Edges $e_1^{i, j} = (v_{2i}^1, v_{2j}^2) (v_{2i}^1, v_{2j + 1}^2)$, $e_2^{i, j} = (v_{2i}^1, v_{2j + 1}^2) (v_{2i + 1}^1, v_{2j + 1}^2)$, $e_3^{i, j} = (v_{2i + 1}^1, v_{2j + 1}^2) (v_{2i + 1}^1, v_{2j}^2)$, $e_4^{i, j} = (v_{2i + 1}^1, v_{2j}^2) (v_{2i}^1, v_{2j}^2)$ for $1 \leq i < r$, $1 \leq j < s$. Again we observe that $e_1^{i, j}$, $e_2^{i, j}$, $e_3^{i, j}$, $e_4^{i, j}$ span a cycle (see \Cref{fig:thm9-4}) and it is balanced since $\sigma(e_1^{i, j}) = \sigma(e_3^{i, j})$ and $\sigma(e_2^{i, j}) = \sigma(e_4^{i, j})$ for $1 \leq i < r$, $1 \leq j < s$. All the cycles are vertex-disjoint.
        \end{enumerate}

        It is easy to observe that all the cycles listed above are vertex-disjoint, so graph $H_1$ has a decomposition into balanced cycles and can be colored using colors $\pm1$. We note that the decomposition consists of $1 + (r - 1) + (s - 1) + (r - 1)(s - 1) = rs$ cycles, so $m(H_1) = 4rs = \frac{m(S)}{2}$ since $m(S) = 4(2r)(2s) / 2 = 8rs$. Clearly, $H_1$ is $2$-regular.

        $V(H_2) = V(S)$, $E(H_2) = E(S) \setminus E(H_1)$. It follows that graph $H_2$ contains edges $e_1^{i, j} = (v_{2i - 1}^1, v_{2j - 1}^2) (v_{2i - 1}^1, v_{2j}^2)$, $e_2^{i, j} = (v_{2i - 1}^1, v_{2j}^2) (v_{2i}^1, v_{2j}^2)$, $e_3^{i, j} = (v_{2i}^1, v_{2j}^2) (v_{2i}^1, v_{2j - 1}^2)$, $e_4^{i, j} = (v_{2i}^1, v_{2j - 1}^2) (v_{2i - 1}^1, v_{2j - 1}^2)$ for $1 \leq i \leq r$, $1 \leq j \leq s$. It is easy to observe that $e_1^{i, j}$, $e_2^{i, j}$, $e_3^{i, j}$, $e_4^{i, j}$ span a cycle and it is balanced since $\sigma(e_1^{i, j}) = \sigma(e_3^{i, j})$ and $\sigma(e_2^{i, j}) = \sigma(e_4^{i, j})$ for $1 \leq i \leq r$, $1 \leq j \leq s$. We also observe that $H_1$ and $H_2$ are edge-disjoint. Graph $H_2$ has a decomposition into vertex-disjoint balanced cycles, so it can be colored using colors $\pm2$. $H_2$ contains $rs$ cycles, so $m(H_2) = 4rs = \frac{m(S)}{2}$.

        Clearly, $S = H_1 \cup H_2$ so $S$ can be colored using $4$ colors. It follows that $\chi'(S) = 4 = \Delta(S)$.
    \end{proof}

    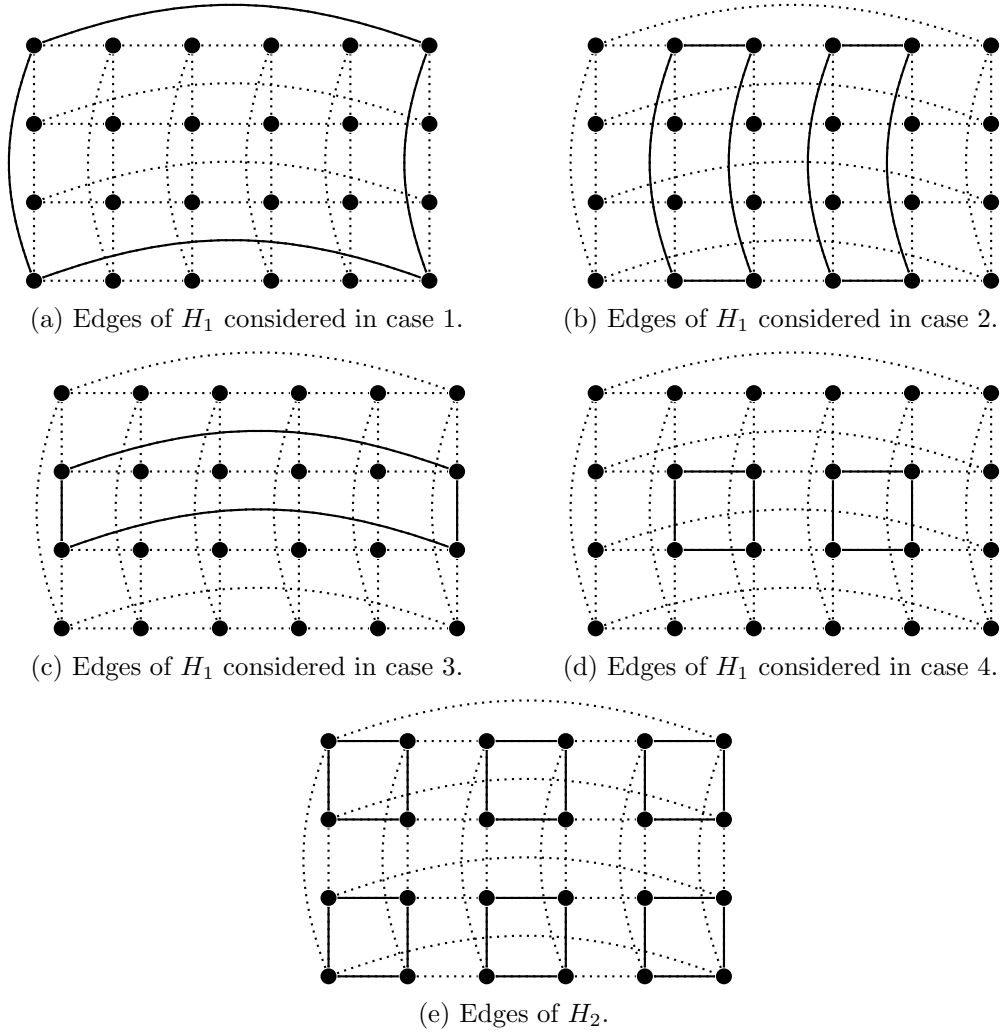
\begin{figure}[htpb]
        \centering
        \begin{subfigure}{0.48\textwidth}
        \begin{tikzpicture}[thick, scale=0.52]
            \node[vertex] (1) at (0, 0) {};
            \node[vertex] (2) at (2, 0) {};
            \node[vertex] (3) at (4, 0) {};
            \node[vertex] (4) at (6, 0) {};
            \node[vertex] (5) at (8, 0) {};
            \node[vertex] (6) at (10, 0) {};
            
            \node[vertex] (7) at (0, 2) {};
            \node[vertex] (8) at (2, 2) {};
            \node[vertex] (9) at (4, 2) {};
            \node[vertex] (10) at (6, 2) {};
            \node[vertex] (11) at (8, 2) {};
            \node[vertex] (12) at (10, 2) {};
            
            \node[vertex] (13) at (0, 4) {};
            \node[vertex] (14) at (2, 4) {};
            \node[vertex] (15) at (4, 4) {};
            \node[vertex] (16) at (6, 4) {};
            \node[vertex] (17) at (8, 4) {};
            \node[vertex] (18) at (10, 4) {};
            
            \node[vertex] (19) at (0, 6) {};
            \node[vertex] (20) at (2, 6) {};
            \node[vertex] (21) at (4, 6) {};
            \node[vertex] (22) at (6, 6) {};
            \node[vertex] (23) at (8, 6) {};
            \node[vertex] (24) at (10, 6) {};

            \foreach \source / \dest in {1/2,2/3,3/4,4/5,5/6,
                                         7/8,8/9,9/10,10/11,11/12,
                                         13/14,14/15,15/16,16/17,17/18,
                                         19/20,20/21,21/22,22/23,23/24}
                    \path [dotted-edge] (\source) -- (\dest);
                    
            \foreach \source / \dest in {1/6,7/12,13/18,19/24}
                    \path [dotted-edge] (\source) edge [bend left=20] (\dest);

            \foreach \source / \dest in {1/7,7/13,13/19,
                                         2/8,8/14,14/20,
                                         3/9,9/15,15/21,
                                         4/10,10/16,16/22,
                                         5/11,11/17,17/23,
                                         6/12,12/18,18/24}
                    \path [dotted-edge] (\source) -- (\dest);
                    
            \foreach \source / \dest in {1/19,2/20,3/21,4/22,5/23,6/24}
                    \path [dotted-edge] (\source) edge [bend left=20] (\dest);

            \foreach \source / \dest in {1/19,19/24,6/24,1/6}
                    \path [edge] (\source) edge [bend left=20] (\dest);

        \end{tikzpicture}
        \caption{Edges of $H_1$ considered in case \ref{case_1_cartesian_c2r_c2s}.}
        \label{fig:thm9-1}
        \end{subfigure}
        \hfill
        \begin{subfigure}{0.48\textwidth}
        \centering
        \begin{tikzpicture}[thick, scale=0.52]
            \node[vertex] (1) at (0, 0) {};
            \node[vertex] (2) at (2, 0) {};
            \node[vertex] (3) at (4, 0) {};
            \node[vertex] (4) at (6, 0) {};
            \node[vertex] (5) at (8, 0) {};
            \node[vertex] (6) at (10, 0) {};
            
            \node[vertex] (7) at (0, 2) {};
            \node[vertex] (8) at (2, 2) {};
            \node[vertex] (9) at (4, 2) {};
            \node[vertex] (10) at (6, 2) {};
            \node[vertex] (11) at (8, 2) {};
            \node[vertex] (12) at (10, 2) {};
            
            \node[vertex] (13) at (0, 4) {};
            \node[vertex] (14) at (2, 4) {};
            \node[vertex] (15) at (4, 4) {};
            \node[vertex] (16) at (6, 4) {};
            \node[vertex] (17) at (8, 4) {};
            \node[vertex] (18) at (10, 4) {};
            
            \node[vertex] (19) at (0, 6) {};
            \node[vertex] (20) at (2, 6) {};
            \node[vertex] (21) at (4, 6) {};
            \node[vertex] (22) at (6, 6) {};
            \node[vertex] (23) at (8, 6) {};
            \node[vertex] (24) at (10, 6) {};

            \foreach \source / \dest in {1/2,2/3,3/4,4/5,5/6,
                                         7/8,8/9,9/10,10/11,11/12,
                                         13/14,14/15,15/16,16/17,17/18,
                                         19/20,20/21,21/22,22/23,23/24}
                    \path [dotted-edge] (\source) -- (\dest);
                    
            \foreach \source / \dest in {1/6,7/12,13/18,19/24}
                    \path [dotted-edge] (\source) edge [bend left=20] (\dest);

            \foreach \source / \dest in {1/7,7/13,13/19,
                                         2/8,8/14,14/20,
                                         3/9,9/15,15/21,
                                         4/10,10/16,16/22,
                                         5/11,11/17,17/23,
                                         6/12,12/18,18/24}
                    \path [dotted-edge] (\source) -- (\dest);
                    
            \foreach \source / \dest in {1/19,2/20,3/21,4/22,5/23,6/24}
                    \path [dotted-edge] (\source) edge [bend left=20] (\dest);

            \foreach \source / \dest in {2/3,20/21,4/5,22/23}
                    \path [edge] (\source) -- (\dest);
            \foreach \source / \dest in {2/20,3/21,4/22,5/23}
                    \path [edge] (\source) edge [bend left=20] (\dest);

        \end{tikzpicture}
        \caption{Edges of $H_1$ considered in case \ref{case_2_cartesian_c2r_c2s}.}
        \label{fig:thm9-2}
        \end{subfigure}
        \\
        \begin{subfigure}{0.48\textwidth}
        \centering
        \begin{tikzpicture}[thick, scale=0.52]
            \node[vertex] (1) at (0, 0) {};
            \node[vertex] (2) at (2, 0) {};
            \node[vertex] (3) at (4, 0) {};
            \node[vertex] (4) at (6, 0) {};
            \node[vertex] (5) at (8, 0) {};
            \node[vertex] (6) at (10, 0) {};
            
            \node[vertex] (7) at (0, 2) {};
            \node[vertex] (8) at (2, 2) {};
            \node[vertex] (9) at (4, 2) {};
            \node[vertex] (10) at (6, 2) {};
            \node[vertex] (11) at (8, 2) {};
            \node[vertex] (12) at (10, 2) {};
            
            \node[vertex] (13) at (0, 4) {};
            \node[vertex] (14) at (2, 4) {};
            \node[vertex] (15) at (4, 4) {};
            \node[vertex] (16) at (6, 4) {};
            \node[vertex] (17) at (8, 4) {};
            \node[vertex] (18) at (10, 4) {};
            
            \node[vertex] (19) at (0, 6) {};
            \node[vertex] (20) at (2, 6) {};
            \node[vertex] (21) at (4, 6) {};
            \node[vertex] (22) at (6, 6) {};
            \node[vertex] (23) at (8, 6) {};
            \node[vertex] (24) at (10, 6) {};

            \foreach \source / \dest in {1/2,2/3,3/4,4/5,5/6,
                                         7/8,8/9,9/10,10/11,11/12,
                                         13/14,14/15,15/16,16/17,17/18,
                                         19/20,20/21,21/22,22/23,23/24}
                    \path [dotted-edge] (\source) -- (\dest);
                    
            \foreach \source / \dest in {1/6,7/12,13/18,19/24}
                    \path [dotted-edge] (\source) edge [bend left=20] (\dest);

            \foreach \source / \dest in {1/7,7/13,13/19,
                                         2/8,8/14,14/20,
                                         3/9,9/15,15/21,
                                         4/10,10/16,16/22,
                                         5/11,11/17,17/23,
                                         6/12,12/18,18/24}
                    \path [dotted-edge] (\source) -- (\dest);
                    
            \foreach \source / \dest in {1/19,2/20,3/21,4/22,5/23,6/24}
                    \path [dotted-edge] (\source) edge [bend left=20] (\dest);

            \foreach \source / \dest in {7/13,12/18}
                    \path [edge] (\source) -- (\dest);
            \foreach \source / \dest in {7/12,13/18}
                    \path [edge] (\source) edge [bend left=20] (\dest);

        \end{tikzpicture}
        \caption{Edges of $H_1$ considered in case \ref{case_3_cartesian_c2r_c2s}.}
        \label{fig:thm9-3}
        \end{subfigure}
        \hfill
        \begin{subfigure}{0.48\textwidth}
        \centering
        \begin{tikzpicture}[thick, scale=0.52]
            \node[vertex] (1) at (0, 0) {};
            \node[vertex] (2) at (2, 0) {};
            \node[vertex] (3) at (4, 0) {};
            \node[vertex] (4) at (6, 0) {};
            \node[vertex] (5) at (8, 0) {};
            \node[vertex] (6) at (10, 0) {};
            
            \node[vertex] (7) at (0, 2) {};
            \node[vertex] (8) at (2, 2) {};
            \node[vertex] (9) at (4, 2) {};
            \node[vertex] (10) at (6, 2) {};
            \node[vertex] (11) at (8, 2) {};
            \node[vertex] (12) at (10, 2) {};
            
            \node[vertex] (13) at (0, 4) {};
            \node[vertex] (14) at (2, 4) {};
            \node[vertex] (15) at (4, 4) {};
            \node[vertex] (16) at (6, 4) {};
            \node[vertex] (17) at (8, 4) {};
            \node[vertex] (18) at (10, 4) {};
            
            \node[vertex] (19) at (0, 6) {};
            \node[vertex] (20) at (2, 6) {};
            \node[vertex] (21) at (4, 6) {};
            \node[vertex] (22) at (6, 6) {};
            \node[vertex] (23) at (8, 6) {};
            \node[vertex] (24) at (10, 6) {};

            \foreach \source / \dest in {1/2,2/3,3/4,4/5,5/6,
                                         7/8,8/9,9/10,10/11,11/12,
                                         13/14,14/15,15/16,16/17,17/18,
                                         19/20,20/21,21/22,22/23,23/24}
                    \path [dotted-edge] (\source) -- (\dest);
                    
            \foreach \source / \dest in {1/6,7/12,13/18,19/24}
                    \path [dotted-edge] (\source) edge [bend left=20] (\dest);

            \foreach \source / \dest in {1/7,7/13,13/19,
                                         2/8,8/14,14/20,
                                         3/9,9/15,15/21,
                                         4/10,10/16,16/22,
                                         5/11,11/17,17/23,
                                         6/12,12/18,18/24}
                    \path [dotted-edge] (\source) -- (\dest);
                    
            \foreach \source / \dest in {1/19,2/20,3/21,4/22,5/23,6/24}
                    \path [dotted-edge] (\source) edge [bend left=20] (\dest);

            \foreach \source / \dest in {8/9,10/11,14/15,16/17,
                                         8/14,9/15,10/16,11/17}
                    \path [edge] (\source) -- (\dest);

        \end{tikzpicture}
        \caption{Edges of $H_1$ considered in case \ref{case_4_cartesian_c2r_c2s}.}
        \label{fig:thm9-4}
        \end{subfigure}
        \\
        \begin{subfigure}{0.48\textwidth}
        \centering
        \begin{tikzpicture}[thick, scale=0.52]
            \node[vertex] (1) at (0, 0) {};
            \node[vertex] (2) at (2, 0) {};
            \node[vertex] (3) at (4, 0) {};
            \node[vertex] (4) at (6, 0) {};
            \node[vertex] (5) at (8, 0) {};
            \node[vertex] (6) at (10, 0) {};
            
            \node[vertex] (7) at (0, 2) {};
            \node[vertex] (8) at (2, 2) {};
            \node[vertex] (9) at (4, 2) {};
            \node[vertex] (10) at (6, 2) {};
            \node[vertex] (11) at (8, 2) {};
            \node[vertex] (12) at (10, 2) {};
            
            \node[vertex] (13) at (0, 4) {};
            \node[vertex] (14) at (2, 4) {};
            \node[vertex] (15) at (4, 4) {};
            \node[vertex] (16) at (6, 4) {};
            \node[vertex] (17) at (8, 4) {};
            \node[vertex] (18) at (10, 4) {};
            
            \node[vertex] (19) at (0, 6) {};
            \node[vertex] (20) at (2, 6) {};
            \node[vertex] (21) at (4, 6) {};
            \node[vertex] (22) at (6, 6) {};
            \node[vertex] (23) at (8, 6) {};
            \node[vertex] (24) at (10, 6) {};

            \foreach \source / \dest in {1/2,2/3,3/4,4/5,5/6,
                                         7/8,8/9,9/10,10/11,11/12,
                                         13/14,14/15,15/16,16/17,17/18,
                                         19/20,20/21,21/22,22/23,23/24}
                    \path [dotted-edge] (\source) -- (\dest);
                    
            \foreach \source / \dest in {1/6,7/12,13/18,19/24}
                    \path [dotted-edge] (\source) edge [bend left=20] (\dest);

            \foreach \source / \dest in {1/7,7/13,13/19,
                                         2/8,8/14,14/20,
                                         3/9,9/15,15/21,
                                         4/10,10/16,16/22,
                                         5/11,11/17,17/23,
                                         6/12,12/18,18/24}
                    \path [dotted-edge] (\source) -- (\dest);
                    
            \foreach \source / \dest in {1/19,2/20,3/21,4/22,5/23,6/24}
                    \path [dotted-edge] (\source) edge [bend left=20] (\dest);

            \foreach \source / \dest in {1/2,7/8,13/14,19/20,
                                         1/7,13/19,2/8,14/20,
                                         3/4,9/10,15/16,21/22,
                                         3/9,4/10,15/21,16/22,
                                         5/6,11/12,17/18,23/24,
                                         5/11,6/12,17/23,18/24}
                    \path [edge] (\source) -- (\dest);

        \end{tikzpicture}
        \caption{Edges of $H_2$.}
        \label{fig:thm9-5}
        \end{subfigure}
        \caption{Example graph $(C_6 \square C_4, \sigma)$ considered in \Cref{cartesian_c2r_c2s}. Edges of respective graphs are marked with solid lines.}
    \end{figure}

    Now let us proceed to products of cycles, where at least one of the cycles is odd.
    To prove the results, we need to recall a lemma: 
    \begin{lemma}[\cite{classes_one_two}]\label{regular_lemma}
        Let $S$ be a $2r$-regular signed graph. $\chi'(S)=\Delta(G)$ if and only if $S$ admits a decomposition into exactly $r$ spanning edge-disjoint $2$-regular balanced subgraphs. \qed
    \end{lemma}
    
    \begin{lemma}\label{odd_num_negative_edges}
        Let $S$ be a $2r$-regular signed graph. If $S$ contains an odd number of negative edges then $\chi'(S) = \Delta(S) + 1$. 
    \end{lemma}
    \begin{proof}
        We assume that $\chi'(S) = \Delta(S)$. From \Cref{regular_lemma} it follows that $S$ admits a decomposition into $r$ spanning edge-disjoint $2$-regular graphs such that their connected components are balanced cycles. We observe that at least one of the graphs must contain an odd number of negative edges and at least one of its subgraphs (which is a cycle) must contain an odd number of negative edges. Obviously, such a cycle is not balanced, so we reach a contradiction.
    \end{proof}

    With these prerequsites in hand, we can proceed with the theorem:
    \begin{lemma}\label{cartesian_c2r_c2s1}
        Let $S_1 = (C_{2r}, \sigma_1)$, $S_2 = (C_{2s+1}, \sigma_2)$ and $S = (G, \sigma) = S_1 \square S_2$. $\mathcal{C}(S) = 1/2$.
    \end{lemma}
    
    \begin{proof}
        We observe that $S$ is a $4$-regular graph. Let us consider $4$ cases:
        \begin{enumerate}
            \item Both $S_1$, $S_2$ are balanced. Then $\chi'(S_1) = 2$ and $\chi'(S_2) = 2$ so it follows from \Cref{cartesian_delta_delta} that $\chi'(S) = \Delta(S)$.

            \item Both $S_1$, $S_2$ are unbalanced. $S_1$ and $S_2$ have $2k + 1$ and $2l + 1$ negative edges, respectively. $S$ has $(2s + 1)(2k + 1) + 2r(2l + 1) = 2(2ks + s + k + 2lr + r) + 1$ negative edges. Since $S$ is $4$-regular, it follows from \Cref{odd_num_negative_edges} that $\chi'(S) = \Delta(S) + 1$.

            \item $S_1$ is unbalanced and $S_2$ is balanced. $S_1$ and $S_2$ have $2k + 1$ and $2l$ negative edges, respectively. $S$ has $(2s + 1)(2k + 1) + 2r2l = 2(2ks + s + k + 2lr) + 1$ negative edges. Again since $S$ is $4$-regular, it follows from \Cref{odd_num_negative_edges} that $\chi'(S) = \Delta(S) + 1$.

            \item\label{case4_cartesian_c2r_c2s1} $S_1$ is balanced and $S_2$ is unbalanced. Let $V(S_1) = \{v_1^1, \ldots, v_{2r}^1\}$ and $V(S_2) = \{v_1^2, \ldots, v_{2s+1}^2\}$.
            
            We will show the decomposition of $S$ into graphs $H_1$, $H_2$ such that both of them can be colored using 2 colors regardless of the edge signs. We first describe the structure of $H_1$. Obviously, $V(H_1) = V(S)$. Graph $H_1$ contains edges:
            \begin{enumerate}
                \item $e_{1}^i = (v_{2i - 1}^1, v_{2s+1}^2) (v_{2i - 1}^1, v_{1}^2)$,
                \item $e_{2}^i = (v_{2i}^1, v_{2s+1}^2) (v_{2i - 1}^1, v_{2s+1}^2)$,
                \item $e_{3}^i = (v_{2i}^1, v_{1}^2) (v_{2i}^1, v_{2s+1}^2)$,
                \item $e_{4}^i =(v_{2i - 1}^1, v_{2s}^2) (v_{2i}^1, v_{2s}^2)$,
                \item ${e'}_{1}^i = (v_{2i - 1}^1, v_{1}^2) (v_{2i - 1}^1, v_{2}^2), \ldots, {e'}_{2s - 1}^i = (v_{2i - 1}^1, v_{2s - 1}^2) (v_{2i - 1}^1, v_{2s}^2)$,
                \item ${e''}_{1}^i = (v_{2i}^1, v_{2}^2) (v_{2i}^1, v_{1}^2), \ldots, {e''}_{2s - 1}^i = (v_{2i}^1, v_{2s}^2) (v_{2i}^1, v_{2s - 1}^2)$
            \end{enumerate}
            for $1 \leq i \leq r$. We observe that those edges span cycles $\gamma_i = (e_{1}^i, e_{2}^i, e_{3}^i, {e''}_{1}^i,$ $\ldots, {e''}_{2s - 1}^i, e_{4}^i, {e'}_{2s}^i, \ldots, {e'}_{1}^i)$ for $1 \leq i \leq r$. All of the cycles $\gamma_i$ are vertex-disjoint and are balanced since $\sigma(e_{1}^i) = \sigma(e_{3}^i)$, $\sigma(e_{2}^i) = \sigma(e_{4}^i)$ and $\sigma({e'}_{j}^i) = \sigma({e''}_{j}^i)$ for $1 \leq j \leq 2s - 1$. It follows that graph $H_1$ has a decomposition into balanced cycles, so it can be colored using colors $\pm1$.

            Graph $H_2$ is such a graph that $V(H_2) = V(S)$ and $E(H_2) = E(S) \setminus E(H_1)$. It follows that graph it contains:
            \begin{enumerate}
                \item Edges $e_{1}^i = (v_{1}^1, v_{i}^2) (v_{2}^1, v_{i}^2), \ldots, e_{2r - 1}^i = (v_{2r - 1}^1, v_{i}^2) (v_{2r}^1, v_{i}^2)$, $e_{2r}^i = (v_{2r}^1, v_{i}^2) (v_{1}^1, v_{i}^2)$ for $1 \leq i \leq 2s - 1$. The edges span $2s - 1$ vertex-disjoint cycles and they are balanced since they are the copies of graph $S_1$ in $S$.

                \item Edges $e_{1}' = (v_{1}^1, v_{2s}^2) (v_{1}^1, v_{2s + 1}^2)$, $e_{2}' = (v_{1}^1, v_{2s + 1}^2) (v_{2r}^1, v_{2s + 1}^2)$, $e_{3}' = (v_{2r}^1, v_{2s + 1}^2) (v_{2r}^1, v_{2s}^2)$, $e_{4}' = (v_{2r}^1, v_{2s}^2) (v_{1}^1, v_{2s}^2)$. Obviously, the edges span a cycle and it is balanced since $\sigma(e_{1}') = \sigma(e_{3}')$, $\sigma(e_{2}') = \sigma(e_{4}')$.

                \item Edges $e_{1}^i = (v_{2i}^1, v_{2s}^2) (v_{2i + 1}^1, v_{2s}^2)$, $e_{2}^i = (v_{2i + 1}^1, v_{2s}^2) (v_{2i + 1}^1, v_{2s + 1}^2)$, $e_{3}^i = (v_{2i + 1}^1, v_{2s + 1}^2) (v_{2i}^1, v_{2s + 1}^2)$, $e_{4}^i = (v_{2i}^1, v_{2s + 1}^2) (v_{2i}^1, v_{2s}^2)$ for $1 \leq i \leq s - 1$. The edges span $s - 1$ vertex-disjoint cycles and they are balanced since $\sigma(e_{1}^i) = \sigma(e_{3}^i)$, $\sigma(e_{2}^i) = \sigma(e_{4}^i)$ for $1 \leq i \leq s - 1$.
            \end{enumerate}

            Clearly, the cycles are vertex-disjoint, and graph $H_2$ has a decomposition into balanced cycles so $H_2$ can be colored using colors $\pm2$. It is easy to observe that graphs $H_1$, $H_2$ are edge-disjoint and are a decomposition of graph $S$. It follows that $S$ can be colored using colors $\pm1$, $\pm2$, so $\chi'(S) = \Delta(S)$.

            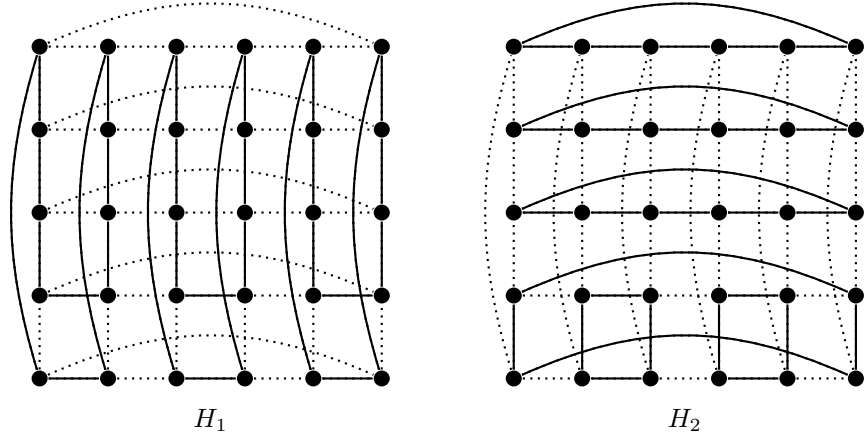
\begin{figure}
                \centering
                \begin{tikzpicture}[thick, xscale=0.45, yscale=0.55]
                    \node[graph-label] (label) at (5,-1) {$H_1$};
                    \node[vertex] (1) at (0, 0) {};
                    \node[vertex] (2) at (2, 0) {};
                    \node[vertex] (3) at (4, 0) {};
                    \node[vertex] (4) at (6, 0) {};
                    \node[vertex] (5) at (8, 0) {};
                    \node[vertex] (6) at (10, 0) {};
                    
                    \node[vertex] (7) at (0, 2) {};
                    \node[vertex] (8) at (2, 2) {};
                    \node[vertex] (9) at (4, 2) {};
                    \node[vertex] (10) at (6, 2) {};
                    \node[vertex] (11) at (8, 2) {};
                    \node[vertex] (12) at (10, 2) {};
                    
                    \node[vertex] (13) at (0, 4) {};
                    \node[vertex] (14) at (2, 4) {};
                    \node[vertex] (15) at (4, 4) {};
                    \node[vertex] (16) at (6, 4) {};
                    \node[vertex] (17) at (8, 4) {};
                    \node[vertex] (18) at (10, 4) {};
                    
                    \node[vertex] (19) at (0, 6) {};
                    \node[vertex] (20) at (2, 6) {};
                    \node[vertex] (21) at (4, 6) {};
                    \node[vertex] (22) at (6, 6) {};
                    \node[vertex] (23) at (8, 6) {};
                    \node[vertex] (24) at (10, 6) {};
                    
                    \node[vertex] (25) at (0, 8) {};
                    \node[vertex] (26) at (2, 8) {};
                    \node[vertex] (27) at (4, 8) {};
                    \node[vertex] (28) at (6, 8) {};
                    \node[vertex] (29) at (8, 8) {};
                    \node[vertex] (30) at (10, 8) {};

                    \foreach \source / \dest in {1/2,2/3,3/4,4/5,5/6,
                                                 7/8,8/9,9/10,10/11,11/12,
                                                 13/14,14/15,15/16,16/17,17/18,
                                                 19/20,20/21,21/22,22/23,23/24,
                                                 25/26,26/27,27/28,28/29,29/30}
                            \path [dotted-edge] (\source) -- (\dest);
                            
                    \foreach \source / \dest in {1/6,7/12,13/18,19/24,25/30}
                            \path [dotted-edge] (\source) edge [bend left=20] (\dest);

                    \foreach \source / \dest in {1/7,7/13,13/19,19/25,
                                                 2/8,8/14,14/20,20/26,
                                                 3/9,9/15,15/21,21/27,
                                                 4/10,10/16,16/22,22/28,
                                                 5/11,11/17,17/23,23/29,
                                                 6/12,12/18,18/24,24/30}
                            \path [dotted-edge] (\source) -- (\dest);
                            
                    \foreach \source / \dest in {1/25,2/26,3/27,4/28,5/29,6/30}
                            \path [dotted-edge] (\source) edge [bend left=20] (\dest);

                    \foreach \source / \dest in {1/2,7/8,
                                                 7/13,13/19,19/25,
                                                 8/14,14/20,20/26}
                            \path [edge] (\source) -- (\dest);
                    \foreach \source / \dest in {3/4,9/10,
                                                 9/15,15/21,21/27,
                                                 10/16,16/22,22/28}
                            \path [edge] (\source) -- (\dest);
                    \foreach \source / \dest in {5/6,11/12,
                                                 11/17,17/23,23/29,
                                                 12/18,18/24,24/30}
                            \path [edge] (\source) -- (\dest);
                    \foreach \source / \dest in {1/25,2/26,3/27,4/28,5/29,6/30}
                            \path [edge] (\source) edge [bend left=20] (\dest);

                \end{tikzpicture}
                \hspace{1cm}%
                \begin{tikzpicture}[thick, xscale=0.45, yscale=0.55]
                    \node[graph-label] (label) at (5,-1) {$H_2$};
                    \node[vertex] (1) at (0, 0) {};
                    \node[vertex] (2) at (2, 0) {};
                    \node[vertex] (3) at (4, 0) {};
                    \node[vertex] (4) at (6, 0) {};
                    \node[vertex] (5) at (8, 0) {};
                    \node[vertex] (6) at (10, 0) {};
                    
                    \node[vertex] (7) at (0, 2) {};
                    \node[vertex] (8) at (2, 2) {};
                    \node[vertex] (9) at (4, 2) {};
                    \node[vertex] (10) at (6, 2) {};
                    \node[vertex] (11) at (8, 2) {};
                    \node[vertex] (12) at (10, 2) {};
                    
                    \node[vertex] (13) at (0, 4) {};
                    \node[vertex] (14) at (2, 4) {};
                    \node[vertex] (15) at (4, 4) {};
                    \node[vertex] (16) at (6, 4) {};
                    \node[vertex] (17) at (8, 4) {};
                    \node[vertex] (18) at (10, 4) {};
                    
                    \node[vertex] (19) at (0, 6) {};
                    \node[vertex] (20) at (2, 6) {};
                    \node[vertex] (21) at (4, 6) {};
                    \node[vertex] (22) at (6, 6) {};
                    \node[vertex] (23) at (8, 6) {};
                    \node[vertex] (24) at (10, 6) {};
                    
                    \node[vertex] (25) at (0, 8) {};
                    \node[vertex] (26) at (2, 8) {};
                    \node[vertex] (27) at (4, 8) {};
                    \node[vertex] (28) at (6, 8) {};
                    \node[vertex] (29) at (8, 8) {};
                    \node[vertex] (30) at (10, 8) {};

                    \foreach \source / \dest in {1/2,2/3,3/4,4/5,5/6,
                                                 7/8,8/9,9/10,10/11,11/12,
                                                 13/14,14/15,15/16,16/17,17/18,
                                                 19/20,20/21,21/22,22/23,23/24,
                                                 25/26,26/27,27/28,28/29,29/30}
                            \path [dotted-edge] (\source) -- (\dest);
                            
                    \foreach \source / \dest in {1/6,7/12,13/18,19/24,25/30}
                            \path [dotted-edge] (\source) edge [bend left=20] (\dest);

                    \foreach \source / \dest in {1/7,7/13,13/19,19/25,
                                                 2/8,8/14,14/20,20/26,
                                                 3/9,9/15,15/21,21/27,
                                                 4/10,10/16,16/22,22/28,
                                                 5/11,11/17,17/23,23/29,
                                                 6/12,12/18,18/24,24/30}
                            \path [dotted-edge] (\source) -- (\dest);
                            
                    \foreach \source / \dest in {1/25,2/26,3/27,4/28,5/29,6/30}
                            \path [dotted-edge] (\source) edge [bend left=20] (\dest);

                    \foreach \source / \dest in {13/14,14/15,15/16,16/17,17/18,
                                                 19/20,20/21,21/22,22/23,23/24,
                                                 25/26,26/27,27/28,28/29,29/30,
                                                 1/7,6/12}
                            \path [edge] (\source) -- (\dest);
                    \foreach \source / \dest in {2/3,3/9,8/9,2/8,
                                                 4/5,5/11,10/11,4/10}
                            \path [edge] (\source) -- (\dest);
                    \foreach \source / \dest in {1/6,7/12,13/18,19/24,25/30}
                            \path [edge] (\source) edge [bend left=20] (\dest);

                \end{tikzpicture}
                \caption{Example graph $(C_6 \square C_5, \sigma)$ considered in \Cref{cartesian_c2r_c2s1}. Edges of $H_1$ and $H_2$ are marked with solid lines.}
            \end{figure}
        \end{enumerate}
        The case $(1)$ corresponds to exactly $1/4$ of possible signatures of $S$ and so the case $(4)$, so it follows that $\mathcal{C}(S) = 2/4 = 1/2$.
    \end{proof}

    \begin{lemma}\label{cartesian_cycle_cycle_switching}
        Let $S_1 = (C_{2r+1}, \sigma_1)$, $S_2 = (C_{2s+1}, \sigma_2)$ and $S = (G, \sigma) = S_1 \square S_2$. If both $S_1$, $S_2$ are unbalanced, then there exists a signed graph $S''$ such that it is a switching equivalent of $S$ and all of its edges are negative.
    \end{lemma}

    \begin{proof}
        Let $V(S_1) = \{ v_1^1, \ldots, v_{2r+1}^1 \}$ and $V(S_2) = \{ v_1^2, \ldots, v_{2s+1}^2 \}$.

        Since $S_1$ is unbalanced it can be switched to graph $S_1'$ such that all its edges are negative. Let $X^1$ denote one of the possible sets of vertices that are switched to get $S_1$ to $S_1'$. Let $X_i^1 = \{ (u,\allowbreak v_i^2) \colon u \in X^1 \}$, for $1 \leq i \leq 2s+1$. We note that for all $i$ it holds that $X_i^1 \subseteq V(S)$. Let $S' = (G, \sigma')$ be a graph formed from $S$ after switching all the vertices from $\bigcup\limits_{i=1}^{2s+1} X_i^1$. It is easy to observe that all the edges $(v_i^1,\allowbreak v_j^2) (v_k^1,\allowbreak v_j^2)$, $i \neq k$ of $S'$ are negative.
        
        We can also show that for all the edges $e = (v_i^1,\allowbreak v_j^2) (v_i^1,\allowbreak v_k^2)$ ($j \neq k$) it holds that $\sigma'(e) = \sigma(e)$. We prove that by contradiction. Let us assume that there is an edge $e = (v_i^1,\allowbreak v_j^2) (v_i^1,\allowbreak v_k^2)$ such that $\sigma'(e) \neq \sigma(e)$. It follows that one of the below cases is true:
        \begin{enumerate}
            \item $(v_i^1,\allowbreak v_j^2) \in X_j^1$ and $(v_i^1,\allowbreak v_k^2) \notin X_k^1$. It follows from the first statement that $v_i^1 \in X^1$ and from the second one that $v_i^1 \notin X^1$, a contradiction.
            \item $(v_i^1,\allowbreak v_j^2) \notin X_j^1$ and $(v_i^1,\allowbreak v_k^2) \in X_k^1$. Clearly, a contradiction.
        \end{enumerate}
        It follows that edge $e$ changed its sign either zero or two times when switching graph $S$ to $S'$ so the sign did not change. It follows that no edge $(v_i^1,\allowbreak v_j^2)$ $(v_i^1,\allowbreak v_k^2)$ ($j \neq k$) changed the sign.

        Analogously, we can define sets $X^2 \subseteq V(S_2) $, $X_1^2$, \ldots, $X_{2r+1}^2$ and show that switching vertices $\bigcup\limits_{i=1}^{2r+1} X_i^2$ in $S'$ changes all the signs of edges in all copies of graph $S_2$ to negative while not changing any signs of all the other edges so it results in a desired signed graph $S''$ with all the edges negative. And since being a switching equivalent is a transitive relation, $S$ and $S''$ are switching equivalent.
    \end{proof}
    
    \begin{lemma}\label{cartesian_c2r1_c2s1}
        Let $S_1 = (C_{2r+1},\allowbreak \sigma_1)$, $S_2 = (C_{2s+1},\allowbreak \sigma_2)$ and $S = (C_{2r+1} \square C_{2s+1},$ $\sigma) = S_1 \square S_2$. Then, $\mathcal{C}(S) = 1/4$.
    \end{lemma}
    
    \begin{proof}
        We observe that $S$ is $4$-regular. We consider $3$ possible cases:
        \begin{enumerate}
            \item\label{cartesian_c2r1_c2s1_case1} Both $S_1$, $S_2$ are balanced. Then $\chi'(S_1) = 2 = \Delta(S_1)$ and $\chi'(S_2) = 2 = \Delta(S_2)$. It follows from \Cref{cartesian_delta_delta} that $\chi'(S) = \Delta(S)$.
            \item\label{cartesian_c2r1_c2s1_case2} Exactly one of $S_1$, $S_2$ is balanced. Without loss of generality, we assume $S_1$ is balanced and $S_2$ is unbalanced, so $S_1$ has $2k$ negative edges and $S_2$ has $2l+1$ of them. We observe that $S$ contains exactly $2k (2s + 1) + (2l+1) (2r + 1) =  2(2ks + k + 2lr + l + r) + 1$ negative edges. It follows from \Cref{odd_num_negative_edges} that $\chi'(S) = \Delta(S) + 1$.
            \item\label{cartesian_c2r1_c2s1_case3} Both $S_1$, $S_2$ are unbalanced. It follows from \Cref{cartesian_cycle_cycle_switching} that $S$ can be switched to graph $S'$ with all edges negative. We assume that $S'$ can be colored using $\Delta(S')$ colors. It follows from \Cref{regular_lemma} that $S'$ has a decomposition into exactly $2$ spanning edge-disjoint $2$-regular balanced subgraphs $H_1$, $H_2$. We observe that $n(S') = (2r + 1) (2s + 1) = 2(2rs + r + s) + 1$ and $m(H_1) = m(H_2) = n(S')$. Graphs $H_1$, $H_2$ have odd numbers of edges. All the edges of $S'$ are negative so both $H_1$ and $H_2$ also have odd numbers of negative edges. Thus it follows from \Cref{odd_num_negative_edges} that $\chi'(H_1) = \Delta(H_1) + 1$ and $\chi'(H_2) = \Delta(H_2) + 1$, so none of $H_1$, $H_2$ are balanced, a contradiction. So $S'$ does not have such a decomposition into $H_1$, $H_2$ and $\chi'(S') = \Delta(S') + 1$. Since $S$ and $S'$ are switching equivalent, $\chi'(S) = \Delta(S) + 1$.
        \end{enumerate}
        The first case corresponds to exactly $1/4$ of possible signatures of $S$ and it is the only case with $\chi'(S)$ being equal to $\Delta(S)$, so it follows that $\mathcal{C}(S) = 1/4$.
    \end{proof}

    Now we covered all the possible Cartesian products of a signed cycle and a signed path and can provide a general theorem:

    \begin{theorem}\label{cartesian_cycles}
        Let $S_1 = (C_{r},\allowbreak \sigma_1)$, $S_2 = (C_{s},\allowbreak \sigma_2)$ be signed cycles and $S = S_1 \square S_2$. Then
        \begin{displaymath}
            \mathcal{C}(S) = 
            \begin{cases} 
                $1$\text{,} & \text{if both } r \text{, } s \text{ are even}\text{,} \\
                $1/2$\text{,} & \text{if one of } r \text{, } s \text{ is even and one is odd}\text{,} \\
                $1/4$\text{,} & \text{if both } r \text{, } s \text{ are odd}\text{.}
            \end{cases}
        \end{displaymath}
    \end{theorem}
    \begin{proof}[Proof of theorem \ref{cartesian_cycles}]
        The first case ($\mathcal{C}(S) = 1$) follows directly from \Cref{cartesian_c2r_c2s}, the second one ($\mathcal{C}(S) = 1/2$)---from \Cref{cartesian_c2r_c2s1}, and the third one ($\mathcal{C}(S) = 1/4$)--- from \Cref{cartesian_c2r1_c2s1}.
    \end{proof}

\section{Tensor products}\label{section_tensor}

    Let $S_1 = (G_1,\allowbreak \sigma_1)$, $S_2 = (G_2,\allowbreak \sigma_2)$ be signed graphs. The tensor product of graphs $S_1$, $S_2$ is a signed graph $S = (G,\allowbreak \sigma)$, such that $V(S) = V(S_1) \times V(S_2)$ and there is an edge $e = (u,\allowbreak u') (v,\allowbreak v')$ in $S$, where $u,\allowbreak v \in V(S_1)$, $u',\allowbreak v' \in V(S_2)$, if and only if $u v \in E(S_1)$ and $u' v' \in E(S_2)$. Then $\sigma(e) = \sigma_1(u v ) \sigma_2(u' v')$.
    
    We denote $S$ by $S_1 \times S_2$. Clearly, $|S| = G_1 \times G_2$.

    First, let us recall a simple relation between the maximum degree of graphs and their tensor product:
    \begin{lemma}\label{lemma_tensor_degrees}
        For arbitrary graphs $G_1$, $G_2$ it holds that $\Delta(G_1 \times G_2) = \Delta(G_1)\allowbreak \Delta(G_2)$.
    \end{lemma}
    \begin{proof}
        From the definition of a tensor product it follows that $\deg_{G_1 \times G_2}((u,\allowbreak u')) = \deg_{G_1}(u) \deg_{G_2}(u')$ for any two vertices $u \in V(G_1)$, $u' \in V(G_2)$. Let $u$, $u'$ be arbitrary vertices of maximum degree in $G_1$, $G_2$, respectively. It is clear that $(u,\allowbreak u')$ is a vertex of maximum degree in $G_1 \times G_2$ and its degree is equal to $\deg_{G_1}(u) \deg_{G_2}(u') = \Delta(G_1) \Delta(G_2)$.
    \end{proof}

    Now we proceed with main results of this section, starting with tensor products of arbitrary graphs and paths on $2$ vertices:
    \begin{theorem}\label{tensor_graph_path}
        \label{thm:tensor-path}
        Let $S_1 = (G_1,\allowbreak \sigma_1)$ be an arbitrary signed graph, $S_2 = (P_2,\allowbreak \sigma_2)$ be a signed path, and $S = S_1 \times S_2$. If $\chi'(S_1) = \Delta(S_1)$, then $\chi'(S) = \Delta(S)$.
    \end{theorem}
    
    \begin{proof}[Proof of theorem \ref{tensor_graph_path}]
        Since $\Delta(S_2) = 1$, it follows from \Cref{lemma_tensor_degrees} that $\Delta(S) = \Delta(S_1)$. We also observe that $m(S) = 2m(S_1)$. Let $c_1$ be an arbitrary optimal coloring of $S_1$, using exactly $\Delta(S_1)$ colors (denoted by $R_1$). For every pair of colors $\pm \alpha$ with $\alpha \ge 0$ (note that $\alpha$ might be possibly equal to $0$) we consider graph $S_1^{\alpha}$ such that $S_1^{\alpha} \subseteq S_1$ and it contains only the edges colored with $\pm \alpha$ in $c_1$. By $S^{\alpha}$ we denote a subgraph of $S$ such that $S^{\alpha} = S_1^{\alpha} \times S_2$. It is easy to observe that $\bigcup_{\pm \alpha \in R_1} S^{\alpha} = S$ and graphs $S^{|\alpha|}$, $S^{|\beta|}$ are edge-disjoint for any $\alpha$, $\beta \in R_1$ such that $|\alpha| \neq |\beta|$.

        We observe that $S_1^{\alpha}$ is a graph of a maximum degree at most $2$ with all the connected components balanced. We show the coloring $c$ of $S$ by showing how to color incidences of $S^{\alpha}$ using only colors used by incidences of $S_1^{\alpha}$. Let $H$ be any connected component of graph $S_1^{\alpha}$. We consider two cases:

        \begin{enumerate}
            \item $H$ is a path. Let $P = H \times S_2$, $P \subseteq S^{\alpha}$. It is easy to observe that $P$ consists of two vertex-disjoint paths isomorphic to $H$, so they can be colored using colors $\pm \alpha$ in the coloring $c$.
            
            \item $H$ is a cycle $C_r$. It follows that $H$ is balanced, so has an even number of negative edges. We consider two cases separately:
            \begin{enumerate}
                \item $r = 2k$. We observe that since $H$ has an even number of edges and an even number of negative edges, it also has an even number of positive edges. Let $C = H \times S_2$. $C$ consists of two vertex-disjoint cycles $(C',\allowbreak \sigma')$, $(C'',\allowbreak \sigma'')$ such that $n(C') = n(C'') = r$. 
                
                We consider two cases:
                \begin{itemize}
                    \item The only edge of $S_2$ is positive. Every edge in $C'$ has the same sign as the corresponding edge in $H$, so there is the same number of negative edges in $C'$ as in $H$, thus $C'$ is balanced.
                    \item The only edge of $S_2$ is negative. Every edge in $C'$ has the opposite sign as the corresponding edge in $H$, so there is the same number of negative edges in $C'$ as the number of positive edges in $H$. $H$ has an even number of positive edges, so $C'$ is balanced.
                \end{itemize}
                The considerations above are exactly the same for $C''$, so we infer that both cycles are balanced and can be colored in $c$ using colors $\pm \alpha$ for some $\alpha > 0$.

                \item $r = 2k + 1$. Let $C = H \times S_2$. It is easy to observe that $C$ is a cycle with $2r$ nodes and there are $r$ pairs of edges with the same signs in $C$. It follows that $C$ is a balanced cycle, so can be colored in $c$ using colors $\pm \alpha$ for some $\alpha > 0$.
            \end{enumerate}
        \end{enumerate}
        We note that if $0 \in R_1$, then $S^0$ is a matching, that is, a collection of paths $P_2$. In this case all $H \times S_2$ consist of two vertex-disjoint paths $P_2$, which of course can be colored with color $0$.
        
        It follows that graph $S^{\alpha}$ can be colored using colors $\pm\alpha$ and since $\bigcup \limits_{\pm \alpha \in R_1} S^{\alpha} = S$ it follows that $S$ can be colored using colors $R_1$, so $\chi'(S) = \Delta(S)$ and $\chi'(S) = \Delta(S)$.
    \end{proof}

    Now let us proceed to the tensor products of arbitrary graphs with trees:
    \begin{lemma}\label{delta_tree}
        For an arbitrary graph $G$ and a tree $T$ it holds that $\Delta(G \times T) = \Delta(T) \Delta(G \times P_2)$.
    \end{lemma}
    \begin{proof}
        It follows from \Cref{lemma_tensor_degrees} that $\Delta(G \times P_2) = \Delta(G)$ and $\Delta(G \times T) = \Delta(T) \Delta(G)$, so $\Delta(G \times T) = \Delta(T) \Delta(G \times P_2)$.
    \end{proof}

    \begin{lemma}\label{lemma_tensor_decomposition_path_times_graph}
        Let $G_1$, $G_2$ be arbitrary graphs, $P_2$ be a path and $G = G_1 \times G_2$, $G_P = G_1 \times P_2$. Then, $G$ has a decomposition into exactly $m(G_2)$ copies of $G_P$. Moreover, every copy may be assigned a corresponding edge in $G_2$ in such a way that two copies have any vertices in common if and only if their respective edges are adjacent in $G_2$.
    \end{lemma}

    \begin{proof}
        The lemma follows directly from the definition of a tensor product. We first consider the case $G_2 = P_2$. In this case it is clear that $G$ has a decomposition into $m(G_2) = 1$ copies of $G_P$, since $G = G_P$. That single copy has a corresponding edge in $G_2$---the only edge of $G_2$.
        
        We observe that when a new edge is added to $G_2$, there appears a new copy of $G_P$ in $G$, denoted by $H'$, so intuitively we assign the new edge in $G_2$ as a corresponding edge for $H'$. It is easy to observe that $H'$ has some vertices in common with other copy $H''$ if and only if the corresponding edges of $H'$, $H''$ are adjacent in $G_2$. It follows from the definition of a tensor product that if copies have vertices in common, there are exactly $\frac{n(G_P)}{2}$ of them.
    \end{proof}

    \begin{theorem}\label{theorem_tensor_graph_tree}
        Let $T$ be a tree, $S_1 = (G_1,\allowbreak \sigma_1)$, $S_2 = (T,\allowbreak \sigma_2)$ and $S = (G_1 \times T, \sigma) = S_1 \times S_2$. If $\chi'(S_1) = \Delta(S_1)$, then $\chi'(S) = \Delta(S)$.
    \end{theorem}
    \begin{proof}[Proof of theorem \ref{theorem_tensor_graph_tree}]
        Let $H = G_1 \times P_2$. It follows from \Cref{lemma_tensor_decomposition_path_times_graph} that $|S|$ has a decomposition into exactly $m(T)$ copies of graph $H$. Moreover, copies of $H$ are connected in the same way as edges of $T$. It means that each edge of $T$ has a respective copy of $H$ in $S$ and two copies of $H$ share a set of vertices if and only if their respective edges are adjacent in $T$.
        
        We observe that graph $S$ can be switched to a switching equivalent in such a way that all the copies of $H$ in the latter signed graph have exactly the same signs. That way we can consider them as copies of some $S' = (H, \sigma')$ -- and without loss of generality we treat these copies of $S'$ as a decomposition of $S$.

        Let $c'$ be an arbitrary optimal edge-coloring of $S'$. Since $S' = (G_1 \times P_2, \sigma')$, from \Cref{thm:tensor-path} we know that $c'$ uses exactly $\Delta(S')$ colors. Let $k = \lfloor \frac{\Delta(S')}{2} \rfloor$. Note that $c'$ uses colors $\{\pm 1,\allowbreak \ldots,\allowbreak \pm k\}$ and additionally color $0$ in case when $\Delta(S')$ is odd.
        
        We define sets $R_1,\allowbreak \ldots,\allowbreak R_{\Delta(T)}$ such that $R_i = \{\pm ((i - 1)k + 1),\allowbreak \ldots,\allowbreak \pm ik\}$ if $c'$ does not use color $0$ (i.e. when $\Delta(S')$ is even), otherwise $R_i = \{0\} \cup \{\pm ((i - 1)k + 1),\allowbreak \ldots,\allowbreak \pm ik\}$. We observe that these sets contain colors that are shifted in such a way that $R_1 \cup \ldots \cup R_\Delta(T)$ contains all the colors $\pm 1,\allowbreak \ldots,\allowbreak \pm k \Delta(T)$ and moreover $R_i \cap R_j \subseteq \{0\}$ for any $i \neq j$.

        We observe that any of the copies of graph $S'$ in $S$ could be colored using any of the $R_1,\allowbreak \ldots,\allowbreak R_{\Delta(T)}$ sets of colors and the coloring of that single copy would be correct.

        Additionally, let $c''$ be an arbitrary optimal edge-coloring of a tree $T$ (without considering its signs). It is a well-known fact that such a coloring uses $\Delta(T)$ colors, namely $\{1,\allowbreak \ldots,\allowbreak \Delta(T)\}$.
        
        Now, we will construct a coloring $c$ of $S$. We first color separately every copy $S'_i$ of $S'$ that we obtained from decomposition of $S$. We remind that by \Cref{lemma_tensor_decomposition_path_times_graph} each such copy has an edge of $T$ assigned to it, which we will denote by $e_i$. This way, we color $S'_i$ using only colors from the $R_{c''(e_i)}$ set and respecting the only requirement that we color all the copies such that corresponding edges of different copies get either the color $0$ or the colors with a difference equal to some multiple of $k$---which by construction of the colors sets is always possible.

        We observe that since $c''$ is a correct coloring of $T$, $c$ is \emph{almost} a correct coloring of $S$---because there may arise conflicts in coloring $c$, but occurring only on edges colored with $0$, since it is the only color that might be shared between different sets $R_i$, $R_j$. It follows that in the case when $\Delta(S')$ is even, color $0$ is not used in $c$ at all and $c$ is a proper coloring. Then $c$ uses exactly $\Delta(T) \Delta(S') = \Delta(S)$ colors.

        Consider now the case when $\Delta(S')$ is odd. Clearly, color $0$ is used in the colorings of all the copies of $S'$ in $S$. By $X \subseteq S$ we denote a subgraph containing all the edges colored by $c$ with color $0$. We will show the procedure of recoloring incidences of $X$ to get a correct coloring of $S$.
        We observe that $X$ is an acyclic graph. If there was a cycle $C_r$ in $X$ then either of two cases would need to hold:
        \begin{enumerate}
            \item There is a copy of $S'$ in $S$ with two adjacent edges both belonging to $X$. It follows that they both were colored with $0$ by $c$. Clearly, such a coloring of that copy is incorrect since the set of edges colored with color $0$ must span a matching, so we reach a contradiction.
            \item There is no copy of $S'$ in $S$ with two adjacent edges both belonging to $X$. It means that any two adjacent edges of $C_r$ in $X$ belong to different copies that share some vertices. We can list these copies: $S'_1$, \ldots, $S'_{r}$, clearly $V(S'_i) \cap V(S'_{i+1}) \neq \emptyset$ for $1 \leq i < r$. Since the cycle must be closed, $V(S'_1) \cap V(S'_r) \neq \emptyset$. We remind that each copy $S'_i$ has a corresponding edge $e_i$ in $T$, so it follows that there is a cycle in $T$ with edges $e_1$, \ldots, $e_r$. Clearly, a contradiction.
        \end{enumerate}

        From the first case above it is also easy to observe that $\Delta(X) = \Delta(T)$. We note that $c$ uses exactly $\Delta(T) (\Delta(S') - 1)$ colors other than $0$. We note that $\Delta(T) (\Delta(S') - 1)$ is even since $\Delta(S')$ is odd. Since signed forests are always $\Delta$-edge-colorable \cite{classes_one_two}, we can easily color $X$ using a symmetric set of $\Delta(T)$ colors that are different than all the colors used by $c$ on the edges of $S$ that do not belong to $X$. That way the new coloring uses exactly $\Delta(T) (\Delta(S') - 1) + \Delta(T) = \Delta(T) \Delta(S') = \Delta(S)$ colors.
    \end{proof}

    Since every path is a tree and every signed path can be edge-colored using $\Delta$ colors, it follows that:
    \begin{cor}\label{tensor_path_path}
        Let $S_1 = (P_r,\allowbreak \sigma_1)$, $S_2 = (P_s,\allowbreak \sigma_2)$ and $S = S_1 \times S_2$. If $r,\allowbreak s > 1$, then $\chi'(S) = \Delta(S)$. \qed
    \end{cor}
    This, together with \Cref{thm:behr-path} and the fact that $P_r \times P_1$ are exactly graphs with no edges proves that the tensor products of paths always belong to class $1^\pm$.

\section{Strong products}\label{section_strong}

    Let $S_1 = (G_1,\allowbreak \sigma_1)$, $S_2 = (G_2,\allowbreak \sigma_2)$ be signed graphs. The strong product of graphs $S_1$, $S_2$ is a signed graph $S = (G,\allowbreak \sigma)$, such that $V(S) = V(S_1) \times V(S_2)$ and there is an edge $e = (u,\allowbreak u') (v,\allowbreak v')$ in $S$, where $u,\allowbreak v \in V(S_1)$, $u',\allowbreak v' \in V(S_2)$, if and only if one of the following conditions holds:
    \begin{enumerate}
        \item $u = v$, $u' v' \in E(S_2)$. In such case $\sigma(e) = \sigma_2(u' v')$;
        \item $u v \in E(S_1)$, $u' = v'$. In such case $\sigma(e) = \sigma_1(u v)$;
        \item $u v \in E(S_1)$ and $u' v' \in E(S_2)$. In such case $\sigma(e) = \sigma_1(u v) \sigma_2(u' v')$.
    \end{enumerate}
    
    We denote $S$ by $S_1 \boxtimes S_2$. Clearly, $|S| = G_1 \boxtimes G_2$. It is clear that $S_1 \boxtimes S_2$ has a decomposition into $S_1 \square S_2$ and $S_1 \times S_2$. We observe that $\deg_S(u,\allowbreak u') = \deg_{S_1 \square S_2}(u,\allowbreak u') + \deg_{S_1 \times S_2}(u,\allowbreak u') = \deg_{S_1}(u) + \deg_{S_2}(u') + \deg_{S_1}(u) \deg_{S_2}(u')$, thus it follows that $\Delta(S) = \Delta(S_1 \square S_2) + \Delta(S_1 \times S_2) = \Delta(S_1) + \Delta(S_2) + \Delta(S_1) \Delta(S_2)$.

    \begin{lemma}\label{strong_lemma}
        Let $S_1$, $S_2$ be signed graphs and $H_1 = S_1 \square S_2$, $H_2 = S_1 \times S_2$, $S = S_1 \boxtimes S_2$. If at least one of $\{\Delta(H_1)$, $\Delta(H_2)\}$ is even, $\chi'(H_1) = \Delta(H_1)$ and $\chi'(H_2) = \Delta(H_2)$, then $\chi'(S) = \Delta(S)$.
    \end{lemma}
    \begin{proof}
        Without loss of generality, we assume that $\Delta(H_2)$ is even.
        
        Let $c_1$, $c_2$ be arbitrary optimal colorings of graphs $H_1$, $H_2$, respectively. $c_1$ uses colors $\pm 1, \ldots, \pm \lfloor \frac{\Delta(H_1)}{2} \rfloor$ and possibly color $0$ if $\Delta(H_1)$ is odd, and $c_2$ uses colors $\pm 1, \ldots, \pm \frac{\Delta(H_2)}{2}$. It is clear that $H_1$, $H_2 \subseteq S$ and $H_1 \cup H_2 = S$. Let $c$ be an edge coloring of $S$ such that:
        \[
            c(u \mbox{:} uv) =
            \begin{cases}
            c_1(u \mbox{:} uv), & \mbox{if } uv \in E(H_1) \mbox{;}
            \\
            c_2(u \mbox{:} uv) + \lfloor \frac{\Delta(H_1)}{2} \rfloor, & \mbox{if } uv \in E(H_2) \mbox{ and } c_2(u \mbox{:} uv) > 0 \mbox{;}
            \\
            c_2(u \mbox{:} uv) - \lfloor \frac{\Delta(H_1)}{2} \rfloor, & \mbox{if } uv \in E(H_2) \mbox{ and } c_2(u \mbox{:} uv) < 0 \mbox{.}
            \end{cases}
        \]
        We observe that incidences of $H_1$ are colored by $c$ using the same colors as by $c_1$, so their coloring is proper. The set of colors used to color incidences of graph $H_2$ has been shifted from $\pm 1, \ldots, \pm \frac{\Delta(H_2)}{2}$ to $\pm (\lfloor \frac{\Delta(H_1)}{2} \rfloor + 1), \ldots, \pm (\lfloor \frac{\Delta(H_1)}{2} \rfloor + \frac{\Delta(H_2)}{2})$. Clearly, incidences of $H_2$ are properly colored in $c$. It is easy to observe that the sets of colors used to color incidences of $H_1$ and $H_2$ in $c$ are disjoint. We consider two cases:
        \begin{enumerate}
            \item $\Delta(H_1)$ is even. Graph $S$ is colored using a set of colors $R = \{\pm 1, \ldots,$ $\pm (\frac{\Delta(H_1)}{2} + \frac{\Delta(H_2)}{2})\}$. Therefore $|R| = 2 (\frac{\Delta(H_1)}{2} + \frac{\Delta(H_2)}{2}) = \Delta(H_1) + \Delta(H_2) = \Delta(S)$, so $\chi'(S) = \Delta(S)$.
            \item $\Delta(H_1)$ is odd. Graph $S$ is colored using a set of colors $R = \{0,\allowbreak \pm 1, \ldots,$ $\pm (\lfloor \frac{\Delta(H_1)}{2} \rfloor + \frac{\Delta(H_2)}{2})\}$. Therefore $|R| = 1 + 2 (\frac{\Delta(H_1) - 1}{2} + \frac{\Delta(H_2)}{2}) = \Delta(H_1) + \Delta(H_2) = \Delta(S)$, so $\chi'(S) = \Delta(S)$.
        \end{enumerate}
    \end{proof}
    
    \begin{theorem}\label{strong_paths}
        Let $S_1 = (P_r,\allowbreak \sigma_1)$, $S_2 = (P_s,\allowbreak \sigma_2)$ and $S = (P_r \boxtimes P_s, \sigma) = S_1 \boxtimes S_2$. If $r > 1$ or $s > 1$, then $\chi'(S) = \Delta(S)$.
    \end{theorem}
    
    \begin{proof}[Proof of theorem \ref{strong_paths}]
        Without loss of generality, we assume that $r \leq s$. Let $H_1 = S_1 \square S_2$, $H_2 = S_1 \times S_2$. We observe that $H_1, H_2$ is a decomposition of $S$ and consider 3 cases:
        \begin{enumerate}
            \item $r = 1$ and $s \geq 2$. Graph $S$ is a path, so can be colored using $\Delta(S)$ colors (see \Cref{thm:behr-path}).
            \item $r = 2$ and $s = 2$. Graph $S$ is a complete graph $K_4$. It follows from \Cref{cartesian_path_path} that incidences of graph $H_1$ can be colored using colors $\pm 1$. Graph $H_2$ is a matching so its incidences can be colored using color $0$. It follows that graph $S$ can be colored using colors $\{0,\allowbreak \pm 1\}$, so $\chi'(S) = 3 = \Delta(S)$.
            \item $r = 2$ and $s > 2$. It follows from \Cref{cartesian_path_path} that incidences of graph $H_1$ can be colored using colors $\{0,\allowbreak \pm 1\}$. It follows from \Cref{tensor_path_path} that incidences of $H_2$ can be colored using colors $\pm 1$, so they can also be colored using $\pm 2$. It follows that graph $S$ can be colored using colors $\{0,\allowbreak \pm 1,\allowbreak \pm 2\}$, so $\chi'(S) = 5 = \Delta(S)$.
            \item $r > 2$ and $s > 2$. We observe that $\Delta(H_1) = \Delta(H_2) = 4$. It follows from \Cref{cartesian_path_path} that $\chi'(H_1) = \Delta(H_1)$ and from \Cref{tensor_path_path} that $\chi'(H_2) = \Delta(H_2)$. It follows from \Cref{strong_lemma} that $\chi'(S) = \Delta(S)$.
        \end{enumerate}
    \end{proof}

\section{Corona products}\label{section_corona}

    Intuitively, a corona product of graphs $G_1$, $G_2$ consists of graph $G_1$, $n(G_1)$ copies of graph $G_2$ and edges connecting the $i$-th vertex of $G_1$ with all the vertices of the $i$-th copy of graph $G_2$ (see an example in \Cref{fig:corona}).
    More formally, the corona product of graphs $G_1$, $G_2$ is a graph $G = G_1 \odot G_2$, such that $V(G) = V(G_1) \cup \bigcup_{i=1}^{n(G_1)} V(G_2^i)$, where $G_2^i$ denotes an $i$th copy of graph $G_2$, and $E(G) = E(G_1) \cup \bigcup_{i=1}^{n(G_1)} E(G_2^i) \cup \{(u, v) \colon u \in V(G_1), v \in V(G_2^i)$, and $u$ is the $i$-th vertex of $G_1 \}$.

    In this section, we consider the edge coloring problem of signed corona products of graphs.
    It turns out that the signature $\sigma$ of a signed corona product $S = (G_1 \odot G_2,\allowbreak \sigma)$ does not matter for the chromatic index of $S$ as far as $\Delta(G_1) \geq 2$.

    \begin{figure}
        \centering
        \begin{tikzpicture}[thick, scale=1.0]
            \node[vertex] (0) at (0, 0) {};
            \node[vertex] (1) at (-1, 2) {};
            \node[vertex] (2) at (0, 3) {};
            \node[vertex] (3) at (1, 2) {};

            \node[vertex] (4) at (3, 0) {};
            \node[vertex] (5) at (2, 2) {};
            \node[vertex] (6) at (3, 3) {};
            \node[vertex] (7) at (4, 2) {};
            
            \node[vertex] (8) at (6, 0) {};
            \node[vertex] (9) at (5, 2) {};
            \node[vertex] (10) at (6, 3) {};
            \node[vertex] (11) at (7, 2) {};
            
            \node[vertex] (12) at (9, 0) {};
            \node[vertex] (13) at (8, 2) {};
            \node[vertex] (14) at (9, 3) {};
            \node[vertex] (15) at (10, 2) {};

            \foreach \source / \dest in
                {1/2,2/3,3/1,
                 5/6,6/7,7/5,
                 9/10,10/11,11/9,
                 13/14,14/15,15/13}
                    \path [edge] (\source) -- (\dest);
            \foreach \source / \dest in
                {0/4,4/8,8/12}
                    \path [edge] (\source) -- (\dest);
            \foreach \source / \dest in
                {0/1,0/2,0/3,
                 4/5,4/6,4/7,
                 8/9,8/10,8/11,
                 12/13,12/14,12/15}
                    \path [dashed-edge] (\source) -- (\dest);

        \end{tikzpicture}
        \caption{An example corona product $P_4 \odot K_3$. Edges of $P_4$ and copies of $K_3$ are marked with solid lines. Edges connecting them are marked with dashed lines.}
        \label{fig:corona}
    \end{figure}
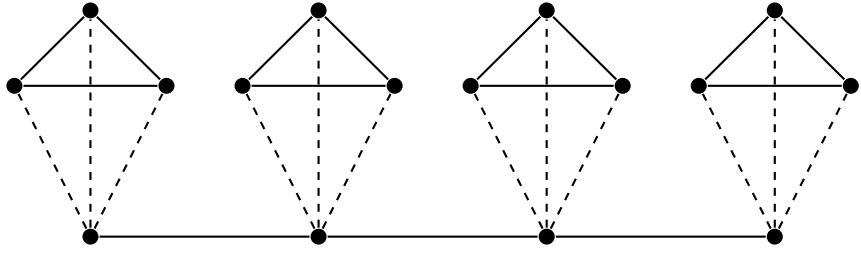

    \begin{theorem}\label{th_corona}
        Let $S = (G_1 \odot G_2,\allowbreak \sigma)$ be a signed corona product of graphs $G_1$ and $G_2$. If $\Delta(G_1) \geq 2$, then $\chi'(S) = \Delta(S)$.
    \end{theorem}

    \begin{proof}[Proof of theorem \ref{th_corona}]
        First, we note that $\Delta(S) = \Delta(G_1) + n(G_2)$.

        Let $v_1, v_2, \ldots, v_{n(G_1)}$ be any sequence of vertices in $G_1$ and $G_{2,1},\allowbreak G_{2,2}, \ldots,\allowbreak G_{2, n(G_1)}$ be a sequence of copies of graph $G_2$ such that $v_i$ is the only vertex connecting graph $G_1$ with graph $G_{2,i}$.
        
        By $S_1$ we denote a subgraph of $S$ that has all the vertices and edges of $G_1$ and corresponding edge signs from $\sigma$.
        Analogously, $S_{2,i}$ corresponds to $G_{2,i}$ with edge signs from $\sigma$. 

        Our main idea is to start from some edge coloring $c_0$ of $S_1$ with colors from $M_{\Delta(S)}$, and then to proceed inductively from $c_{i - 1}$ to $c_i$ ($i = 1, \ldots, n(S_1)$) by coloring a graph induced by $v_i$ and $S_{2, i}$ while preserving the invariant that we use only colors from $M_{\Delta(S)}$.
        Since we will show that such coloring can be always found, from now on without loss of generality we may assume that $G_1$ is a regular graph.

        Clearly, for $c_0$ the invariant holds, since first we can find an edge coloring $c^\otimes$ of $S_1$ using $\Delta(S_1) + 1 \le \Delta(S_1) + n(G_2) = \Delta(S)$ colors.
        Now, we have two possible cases:
        \begin{itemize}
            \item either $M_{\Delta(S_1) + 1} \subseteq M_{\Delta(S)}$, so we set $c_0$ as $c^\otimes$,
            \item or $M_{\Delta(S_1) + 1} \nsubseteq M_{\Delta(S)}$, thus $0 \in M_{\Delta(S_1) + 1}$, but $0 \notin M_{\Delta(S)}$ -- and then we define
            \begin{align*}
                c_0(v \text{:} e) =
                \begin{cases}
                    \pm \left\lfloor\frac{\Delta(S)}{2}\right\rfloor & \text{if $c^\otimes(v\text{:}e) = 0$,} \\
                    c^\otimes(v \text{:} e) & \text{otherwise.}
                \end{cases}
            \end{align*}
            Note that if $0 \notin M_{\Delta(S)}$, then $\Delta(S)$ is even and the resulting coloring is correct, since $\pm \frac{\Delta(S)}{2} \notin M_{\Delta(S_1) + 1}$, and we can always consistently assign signs in the formula above, as $0$-edges always form a matching so any alignment of signs consistent with $\sigma$ suffices.
        \end{itemize}
        
        Now, let us also fix $c^*_i$ as an edge coloring of $S_{2, i} + v_i$ with $n(G_2) + 1 = \Delta(S_{2,i} + v_i) + 1$ colors.
        Since $n(G_2) + 1 < \Delta(S_1) + n(G_2) = \Delta(S)$, we can transform this coloring to $c'_i$ as following:
        \begin{itemize}
            \item either $M_{n(G_2) + 1} \subseteq M_{\Delta(S)}$, so we set $c'_i$ as $c^*_i$,
            \item or $M_{n(G_2) + 1} \nsubseteq M_{\Delta(S)}$, thus $0 \in M_{n(G_2) + 1}$, but $0 \notin M_{\Delta(S)}$ -- and then we define
            \begin{align*}
                c'_i(v \text{:} e) =
                \begin{cases}
                    \pm \left\lfloor\frac{\Delta(S)}{2}\right\rfloor & \text{if $c^*_i(v\text{:}e) = 0$,} \\
                    c^*(v \text{:} e) & \text{otherwise.}
                \end{cases}
            \end{align*}
            As with $c^\otimes$ and $c_0$ above, since $0 \notin M_{\Delta(S)}$, we know that $\Delta(S)$ is even and the resulting coloring is correct, since $\pm \left\lfloor\frac{\Delta(S)}{2}\right\rfloor \notin M_{n(G_2) + 1}$.
        \end{itemize}

        Observe that any $c'_i$ has the following property $(I)$:
        \begin{itemize}
            \item if $0 \in M_{\Delta(S)}$, then there is at most one $k$ such that there is an incidence of $v_i$ with color $k$ and no incidence of $v_i$ with color $-k$,
            \item if $0 \notin M_{\Delta(S)}$, then there are at most two such values of $k$ -- and for one of them $\pm k$ induces a matching in $G_2$.
        \end{itemize}
        The last part stems from the fact that the second missing color can arise only when $M_{n(G_2) + 1} \nsubseteq M_{\Delta(S)}$ by using $\pm \left\lfloor\frac{\Delta(S)}{2}\right\rfloor$ in place of $0$.
        Note that we have $\deg_{S_1}(v_i) = \Delta(S_1)$. Therefore, by our construction of $c_0$ the property $(I)$ holds for all vertices $v_i$. 

        For any $i > 0$ we just need to construct $c_i$ as a union of $c_{i - 1}$ and a proper recoloring of $c'_i$ ensuring that all pairs of colors from $c'_i$ are recolored to all $\pm k \in M_{\Delta(S)}$ (including a ``fake pair'' $0 \in M_{\Delta(S)}$ if necessary) not used by incidences already colored in $c_{i - 1}$. Moreover, we have to ensure that these colors are used for both the incidences of $v_i$ without introducing any conflicts with $c_{i - 1}$.

        \textbf{Case 1:} First, let us consider the case that $0 \in M_{\Delta(S)}$. From the property $(I)$ we know that there exists at most one value $l \neq 0$ such that one of the incidences of $v_i$ in $c_{i - 1}$ has color $l$, but none of them has color $-l$, and similarly for $c'_i$.
        
        \textbf{Case 1a:} If $\Delta(S_1)$ is even, then $\Delta(S_{2, i} + v_i)$ has to be odd. Thus, $c^*_i$ and $c'_i$ by construction do not use color $0$ at all. However, this means that the incidences of $v_i$ in $c'_i$ have exactly colors $\pm k'_1$, \ldots, $\pm k'_{t'}$, and $l'$ (i.e. there is one color without pair at $v_i$ in $c'_i$) with all $k'_j \neq 0$.

        \textbf{Case 1a':} Suppose now that the incidences of $v_i$ in $c_{i - 1}$ have colors $\pm k_1$, \ldots, $\pm k_t$, $l$. Observe that since $\Delta(S_1)$ is even, there is some $j = 1, \ldots, t$ such that $k_j = 0$.
        Let us map (a) colors $\pm l'$ from $c'_i$ to $\mp l$ (reverting the signs) and (b) all colors $\pm k'_s$ from $c'_i$ for $s = 1, \ldots, t'$ to some pairs from $M_{\Delta(S)}$ different than all $\pm k_1$, \ldots, $\pm k_t$ and $\pm l$.

        \textbf{Case 1a'':} If, on the other hand, the incidences of $v_i$ in $c_{i - 1}$ have colors $\pm k_1$, \ldots, $\pm k_t$, then all such pairs of colors have to be non-zero. Moreover, $\Delta(S_1) \ge 2$ ensures that $t \ge 1$. Thus, it is sufficient to map
        \begin{itemize}
            \item colors $\pm l'$ from $c'_i$ to $\pm k_1$ for all edges not incident to $v_i$,
            \item the incidence of $v_i$ with color $l'$ in $c'_i$ to $0$,
            \item all colors $\pm k'_s$ from $c'_i$ for $s = 1, \ldots, t'$ to some pairs from $M_{\Delta(S)}$ different than all $\pm k_1$, \ldots, $\pm k_t$.
        \end{itemize}
        
        \textbf{Case 1b:} When $\Delta(S_1)$ is odd, then $c_0$ by construction does not use color $0$ at all, therefore all incidences of $v_i$ in $c_{i - 1}$ have non-zero colors -- in particular, they have colors $\pm k_1$, \ldots, $\pm k_t$, $l$ with all $k_j \neq 0$.

        \textbf{Case 1b':} Suppose that the incidences of $v_i$ in $c'_i$ have colors $\pm k'_1$, \ldots, $\pm k'_{t'}$, $l'$. Then, since $\Delta(S_{2, i} + v_i)$ is even, it has to be the case that $k'_j = 0$ for some $j = 1, \ldots, t'$. In this case, we proceed as in the case 1a', additionally trivially copying incidences with color $0$ from $c'_i$ to $c_i$.

        \textbf{Case 1b'':} If, on the other hand, the incidences of $c'_i$ have only colors $\pm k'_1$, \ldots, $\pm k'_{t'}$, then $k'_j \neq 0$ for all $j = 1, \ldots, t'$. Then we map
        \begin{itemize}
            \item colors $\pm k'_1$ from $c'_i$ to $\pm k_1$ for all edges not incident to $v_i$,
            \item the incidence of $v_i$ with colors $\pm k'_1$ in $c'_i$ to $0$ and $-l$, respectively,
            \item all possible incidences with color $0$ in $c'_i$ to $\pm l$ so that they do not conflict with the one of incidences above -- note that $0$ might be only possibly used in $S_{2, i} + v_i$, but not at $v_i$ in $c'_i$,
            \item all colors $\pm k'_s$ from $c'_i$ for $s = 2, \ldots, t'$ to some pairs from $M_{\Delta(S)}$ different than all $\pm k_1$, \ldots, $\pm k_t$ and $\pm l$.
        \end{itemize}
        A simple degree count for $v_i$ shows that such recolorings are always possible: the number of colors used at $v_i$ by both $c_{i - 1}$ and $c'_i$ has to be equal to $\deg_S(v_i) = \Delta(S) = |M_{\Delta(S)}|$ -- and it is easy to check that in each case $c_i$ uses exactly the colors from $M_{\Delta(S)}$.
        Moreover, it can be verified that we map all the colors from $c'_i$ (i.e. all appearing at $v_i$ plus one other) and that the resulting $c_i$ is a proper edge coloring of a signed graph since the property $(I)$ ensures that $l, l' \neq 0$.

        Note that in general for $k'_j = 0$, it is always possible to map it to a pair of values $\pm k_i$ for some $k_i \neq 0$. And the construction ensures that we would never try doing the reverse, that is, mapping a true pair of colors to $0$.

        \textbf{Case 2:} for $0 \notin M_{\Delta(S)}$ suppose that the incidences of $v_i$ in $c_{i - 1}$ have colors $\pm k_1$, \ldots, $\pm k_t$, $l_1$, $l_2$. Since there is no unusual, one-element pair $\pm 0$ and $\deg_S(v_i) = \Delta(S) = |M_{\Delta(S)}|$ is even, it holds that either the incidences of $v_i$ in $c'_i$ have colors $\pm k'_1$, \ldots, $\pm k'_{t'}$, $l'_1$, $l'_2$ or just $\pm k'_1$, \ldots, $\pm k'_{t'}$.
        In the first case, we map $\pm l'_1$ and $\pm l'_2$ from $c'_i$ to $\mp l_1$ and $\mp l_2$ (reverting the signs), respectively, and $\pm k'_j$ from $c'_i$ to available pairs as above.
        In the second case $t' \ge 1$, so we recolor $\pm k'_1$ to $\{-l_1, -l_2\}$ and all other $k'_s$ as before.
        
        Similarly, when there are only $l_1$ and $l'_1$, but not $l_2$ or $l'_2$ it is sufficient to map $\pm l'_1$ to $\mp l_1$ (reverting the signs).

        And finally, the case when there is no $l_1$ or $l_2$:
        \begin{itemize}
            \item if there is $l'_1$ and $l'_2$, then according to the property $(I)$ $\pm l'_2$ induces a matching in $S_{2, i} + v_i$. We can map colors $\pm l'_1$ from $c'_i$ for all edges not incident to $v_i$ to $\pm k_1$ and map all $\pm k'_s$ from $c'_i$ to some available pairs, other than all $\pm k_j$ ($j = 1, \ldots, t$). Finally, we have a matching plus a single edge, that is, an acyclic graph, so we can safely map it to the remaining available pair.
            \item if there is neither $l'_1$ nor $l'_2$, then we just map $\pm k'_s$ to available pairs, other than all $\pm k_j$  ($j = 1, \ldots, t$).
        \end{itemize}

    Therefore, in every case it is possible to obtain $c_i$ from $c_{i - 1}$ which directly implies that the required $\Delta(S)$-edge coloring of the whole signed corona product graph $S$ always exists.
    \end{proof}

    We also observe that a corona product of graphs $G_1$, $G_2$ with $\Delta(G_1) \le 1$ is either a subgraph of $(1)$ a collection of disjoint cliques (when $G_1$ has no edges) or $(2)$ a collection of cliques, either disjoint, or connected in pairs with a single edge (when $G_1$ is a matching). However, these two problems are still open:
    \begin{con}
        Given a signed complete graph $S = (K_n, \sigma)$, show that $\chi'(S) = \Delta(S) = n - 1$.
    \end{con}

    \begin{con}
        Given a signed graph consisting of two identical cliques connected by a single edge $S = ((K_n \cup K_n) + e, \sigma)$, show that $\chi'(S) = \Delta(S) = n$.
    \end{con}

    Note that in the latter if all the edge signs were consistent on both copies of $K_n$, then the conjecture would hold. It would be sufficient to find $c'$ as $n$-edge colorings of both copies of $K_n$, such that $\alpha$ is a missing color both for $v_1$ in $c'$ applied for the first $K_n$ and for $v_2$ in $c'$ applied for the second $K_n$ for the joining edge $e = v_1 v_2$. Then it is enough to return either a join of $c'$ for both $K_n$ with an appropriate color on both endpoints of $e$ (if the sign of $e$ is positive), or a join of $c'$, $-c'$ for each $K_n$ and a proper coloring of $e$ (if the sign of $e$ is negative).

\end{document}